\documentclass[11pt]{amsart} \textwidth=14.5cm \oddsidemargin=1cm
\usepackage{setspace}
\usepackage{amsthm}
\usepackage[rgb]{xcolor}
\usepackage{times}
\usepackage[linktocpage=true]{hyperref}
\usepackage{mathrsfs}

\usepackage{xypic}
\usepackage{tikz}
\usepackage{verbatim}
\usepackage{amscd}
\usepackage{amsmath}
\usepackage{amsfonts}
\usepackage{amsxtra}
\usepackage{amssymb}
\usepackage{eucal}
\usepackage{latexsym,todonotes}
\usepackage{stmaryrd}
\usepackage{graphicx}%
\usepackage{dynkin-diagrams}
\usepackage{tikz}
\usetikzlibrary{intersections}
\newtheorem{theorem}{Theorem}[section]

\newtheorem{definition}[theorem]{Definition}
\newtheorem{example}[theorem]{Example}

\newtheorem{lemma}[theorem]{Lemma}
\newtheorem{lem}[theorem]{Lemma}

\newtheorem{proposition}[theorem]{Proposition}
\newtheorem{remark}[theorem]{Remark}

\numberwithin{equation}{section}

\newcommand{\Z}{\mathbb{Z}}

\newcommand{\Q}{\mathbb{Q}}
\newcommand{\C}{\mathbb{C}}

\newcommand{\Hy}{\mathscr{H}}

\newcommand{\End}{\text{End}\ }
\newcommand{\gl}{\mathfrak g \mathfrak l}

\newcommand{\V}{\mathbb V}
\newcommand{\W}{\mathbb {W}}

\newcommand{\io}{\imath}

\newcommand{\U}{\mathbf U}
\newcommand{\Ui}{\mathbf U^{\imath}}
\newcommand{\va}{\varsigma}

\newcommand{\so}{\mathfrak s \mathfrak o}
\newcommand{\spin}{\mathfrak s \mathfrak p}
\newcommand{\Qy}{e}
\newcommand{\B}{\mathfrak B}
\newcommand{\ty}{\mathfrak b}

\allowdisplaybreaks[4]

\begin{document}
\title[Canonical bases of $q$-Brauer algebras and $\imath$Schur dualities]
{Canonical basis of $q$-Brauer algebras and $\imath$Schur dualities}

\author[Weideng Cui]{Weideng Cui}
\address{School of Mathematics, Shandong University, Jinan, Shandong 250100, China}
\email{cwdeng@amss.ac.cn (Cui)}

\author[Yaolong Shen]{Yaolong Shen}
\address{Department of Mathematics\\ University of Virginia\\ Charlottesville, VA 22904}
\email{ys8pfr@virginia.edu (Shen)}


\begin{abstract}
Expanding the classical work of Kazhdan-Lusztig, we construct a bar involution and canonical bases on the $q$-Brauer algebra introduced by Wenzl. We define explicit actions of the $q$-Brauer algebra on the tensor spaces, and formulate $\imath$Schur dualities between the $q$-Brauer algebra and the $\imath$quantum groups of type AI and AII respectively.
\end{abstract}

\maketitle
\setcounter{tocdepth}{1}
\tableofcontents

\section{Introduction}

In the classical Schur duality, the actions of the general linear group $GL_m$ and the symmetric group $\mathfrak S_n$ on the tensor space $(\C^m)^{\otimes n}$ commute with each other and satisfy the double centralizer property. A quantum analog of this duality is provided by the quantum group $\U_q(\gl_m)$ and the type $A$ Iwahori-Hecke algebra $\Hy_{\mathfrak S_n}$; see \cite{Jim86}. Moreover, the type $A$ (parabolic) Kazhdan-Lusztig basis (\cite{KL79}, \cite{De87}) can be identified with the canonical basis on the tensor product of the natural representation of the type $A$ quantum group via the Schur-Jimbo duality  \cite{FKK98} (cf. \cite{LW20}).

In \cite{Br37} Brauer introduced the so-called Brauer algebra, and established the double centralizer property between it and the orthogonal group $O_m$ or symplectic group $Sp_{2m}$. The Brauer algebra was further studied in \cite{Br56a}, \cite{Br56b} and so on. The Birman-Murakami-Wenzl algebra (or BMW algebra for short), as a two-parameter deformation of the Brauer algebra, was algebraically defined by Birman and Wenzl \cite{BW89}, and independently by Murakami \cite{Mu87}. In the Schur-Jimbo duality, when $\U_q(\gl_m)$ is replaced by $\U_q(\mathfrak o_m)$ or $\U_q(\spin_{2m})$, the role of $\Hy_{\mathfrak S_n}$ was played by the BMW algebra with the parameters being appropriately specialized; see \cite{CP94} or \cite{Ha92}. In \cite{FG95} a canonical basis of the BMW algebra has been constructed and the associated cell structure has been studied. However, the Iwahori-Hecke algebra $\Hy_{\mathfrak S_n}$ is not naturally a subalgebra of the BMW algebra while the algebras $\U_q(\mathfrak o_m)$ and $\U_q(\spin_{2m})$ are not isomorphic to subalgebras of the type $A$ quantum groups either.

Besides the BMW algebra, another multi-parameter deformation of the Brauer algebra, which depends on two indeterminates $q$ and $z$, was introduced by Molev \cite{M03}; moreover, Molev showed that the action of his algebra (specializing $z$ to $q^{m}$) on $\V^{\otimes n}$ commutes with that of the twisted quantized enveloping algebra $\U_{q}^{\mathrm{tw}}(\mathfrak{so}_{m})$ introduced by Noumi in \cite{No96}, where $\V$ is the natural representation of $\U_{q}(\mathfrak{sl}_{m})$ (also cf. \cite{We12b}). 

Later on, in \cite{We12a}, Wenzl defined a quotient of Molev's algebra called the {\em $q$-Brauer algebra}, which will be the main object of this paper as its defining relations are more friendly to work with. Many properties of the $q$-Brauer algebra have been studied by Nguyen. For example, in \cite{N14}, Nguyen constructed a standard basis for the $q$-Brauer algebra which is labeled by a natural basis of Brauer algebras; the $q$-Brauer algebra contains $\Hy_{\mathfrak S_n}$ as a natural subalgebra under the standard basis. Moreover, in \cite{N14} and \cite{N18}, it was shown that the $q$-Brauer algebra is a cellular algebra and its irreducible representations can be classified using the general theory of cellular algebras in \cite{GL96}.


The aforementioned twisted quantized enveloping algebra $\U_{q}^{\mathrm{tw}}(\mathfrak{so}_{m})$ can be regarded as an example of the $\imath$quantum group arising from so-called quantum symmetric pairs of type AI; see \cite[Section 6]{Let99}. Given an involution $\theta$ on a complex simple Lie algebra $\mathfrak{g}$, we can obtain a symmetric pair $(\mathfrak{g}, \mathfrak{g}^{\theta})$, or a pair of enveloping algebras $(\mathbf{U}(\mathfrak{g}), \mathbf{U}(\mathfrak{g}^{\theta}))$, where $\mathfrak{g}^{\theta}$ denotes the fixed point subalgebra. The pairs $(\mathfrak{sl}_{m}, \mathfrak{so}_{m})$ and $(\mathfrak{sl}_{2m}, \mathfrak{sp}_{2m})$ are examples of symmetric pairs. The classification of symmetric pairs of finite type is equivalent to the classification of real simple Lie algebras, which can be described in terms of the Satake diagrams. Letzter \cite{Let99}, \cite{Let02} systematically developed the theory of quantum symmetric pairs $(\mathbf{U}, \mathbf{U}^{\imath})$ of finite type as a quantization of $(\mathbf{U}(\mathfrak{g}), \mathbf{U}(\mathfrak{g}^{\theta}))$. The algebra $\U^\io$ will be referred to as an $\io$quantum group. Kolb \cite{Ko14} has further studied and generalized this theory to the Kac-Moody case.

In recent years, Bao and Wang have generalized Lusztig's approach on canonical bases in \cite{Lu94} and developed a general theory of the canonical basis for $\imath$quantum groups arising from quantum symmetric pairs of arbitrary finite type in \cite{BW18b} and of Kac-Moody type in \cite{BW21}. They showed that any based module of a quantum group of finite type  (cf. \cite[Chapter 27]{Lu94}), when viewed as a module over an $\imath$quantum group, is endowed with a new bar involution and a unique bar-invariant basis (called the $\imath$canonical basis).

Bao and Wang have also shown in \cite{BW18a} that the Iwahori-Hecke algebra of type $B$ and the $\imath$quantum group of type AIII satisfy a double centralizer property which generalizes the Schur-Jimbo duality  (also cf. \cite{B17}); moreover, the Kazhdan-Lusztig basis of type $B$ was shown to coincide with the $\imath$canonical basis on the tensor product module of the $\imath$quantum group. Later on, this $\imath$Schur duality was further generalized to a multi-parameter setting in \cite{BWW18}. More recently, in \cite{SW21}, a common generalization of Schur dualities of both types $A$ and $B$ has been constructed.

In this paper, we first study the $q$-Brauer algebra $\B_n(q,z)$ and define a bar involution on it. The bar involution is shown to be compatible with the one on its natural subalgebra $\Hy_{\mathfrak S_n}$. Applying the bar involution to the standard basis of $\B_n(q,z)$ constructed in \cite{N14}, we are able to construct a Kazhdan-Lusztig-type basis (called the canonical basis) on $\B_n(q,z)$ through a standard approach due to Lusztig. A direct consequence of the compatibility of the bar involutions is that the usual type $A$ Kazhdan-Lusztig basis is a part of the canonical basis we obtain. Moreover, one can see that the coefficients, when expanding the canonical basis as a sum of the standard basis elements, are polynomials in $q$, which do not depend on $z$. A similar phenomenon was found in \cite[\S 5.2]{FG95}.

Enlightened by the pioneering work about $\imath$Shur dualities in \cite{BWW18}, we then construct explicit actions of the $q$-Brauer algebra $\B_n(q,z)$ on the tensor product modules of the $\io$quantum groups of type AI and AII respectively. We show that the actions we define indeed commute with the natural actions of the $\io$quantum groups (allowing any parameter) and prove the double centralizer property in both cases. The double centralizer property follows from taking a $q\to 1$ limit and classical results in \cite{Br37} (cf. \cite{Br56a}, \cite{Br56b}), as the $\imath$quantum group reduces to the enveloping algebra of the special orthogonal (resp. symplectic) Lie algebra. The commuting action in the case of type AI was also formulated in \cite{M03} with a restriction on the parameters of the $\imath$quantum group and using the $R$-matrix presentation of the quantum group $\U_q(\gl_m)$. Both commuting actions were also discovered in \cite[(7.10)--(7.11)]{ST19} through the Web category but not explicitly constructed.



To further extend this work, it remains an open question whether the $\imath$canonical basis (for suitable $z$) on the tensor product module can be realised and reconstructed purely through the $q$-Brauer algebra actions. In \cite{W21}, Wang has proposed a positivity conjecture for the $\imath$canonical basis of (quasi-)split ADE-type. It makes sense to formulate a conjecture that the canonical basis of the $q$-Brauer algebra $\B_n(q,z)$ we construct has the positivity as well. In this paper some small rank cases (when $n=2,3$) have been checked to support the conjecture, that is, one can see that the canonical basis and the structure constants with respect to it have certain positivity in those cases.


The paper is organized as follows. In \S \ref{sec:Brauer}, we recall the definition of the $q$-Brauer algebra $\B_n(q,z)$ and the construction of its standard basis. In \S \ref{sec:bar}, we construct a bar involution and a canonical basis for $\B_n(q,z)$. In \S \ref{sec:AI}--\S\ref{sec:AII}, we formulate $\imath$Schur dualities between $\B_n(q,z)$ (with $z$ being appropriately specialized) and the $\imath$quantum groups of type AI and AII, respectively. In more detail, in \S \ref{sec:AI}, we define commuting actions on $\V^{\otimes n}$ between $\B_n(q,q^{m})$ and the $\imath$quantum group $\U^\io(\so_{m})$ of type AI and prove that they satisfy the double centralizer property. In \S \ref{sec:AII}, we deal with the case of the $\imath$quantum group of type AII.



%
%

\vspace{2mm}
\noindent {\bf Acknowledgement.}
We thank Weiqiang Wang for encouraging the collaboration and providing insightful advice. We also thank an anonymous referee for helpful suggestions. W.~Cui is partially supported by Young Scholars Program of Shandong University, Shandong Provincial Natural Science Foundation (Grant No. ZR2021MA022) and the NSF of China (Grant No. 11601273). Y.~Shen is supported by a GSAS fellowship at University of Virginia and WW's NSF Graduate Research Assistantship (DMS-2001351).


\section{$q$-Brauer algebras}
\label{sec:Brauer}
\subsection{Brauer algebras}
Fix an integer $N\in\mathbb{Z}\setminus \{0\}$. Let $D_{n}(N)$ denote the Brauer algebra which is a $\mathbb{Z}$-algebra with a linear basis consisting of all partitions of the set $$\{1,2,\ldots,n,1',2',\ldots,n'\}$$ into two-element subsets. As usual, we can represent each basis element by a diagram with two rows, where the top row has $n$ vertices marked by $1,2,\ldots,n$, and the bottom row is numbered by $1',2',\ldots,n'$; the vertex $i$ is joined to $j$ by an edge if they are in the same subset. We will call an edge {\em horizontal} if it connects two vertices on the same row, and {\em vertical} otherwise. Two diagrams $d_1$ and $d_2$ are multiplied by concatenation, that is, $d_1\cdot d_2$ is defined to be $N^{\gamma(d_1, d_2)}d$, where $\gamma(d_1, d_2)$ counts the number of cycles produced by forming the concatenation and $d$ is the resulting diagram after removing all cycles.

In fact, we have the following presentation for the Brauer algebra $D_{n}(N)$.
\begin{definition}{\rm (cf. \cite[$\S$2.1.1]{N14}) }\label{brau}
The Brauer algebra $D_{n}(N)$ is the unital associative $\mathbb{Z}$-algebra generated by $s_1,\ldots, s_{n-1}$, together with elements $e_{(1)}, e_{(2)}, \ldots, e_{(\lfloor\frac{n}{2}\rfloor)}$, which satisfy the following relations:
\begin{equation*}
\begin{aligned}
&(S_{1})\ s_i^{2}=1&&\mbox{for $1\leq i\leq n-1$,}\\[0.1cm]
&(S_{2})\ s_is_{i+1}s_i=s_{i+1}s_is_{i+1}&&\mbox{for $1\leq i\leq n-2$,}\\[0.1cm]
&(S_{3})\ s_is_j=s_js_i&&\mbox{for $|i-j|\geq 2$,}\\[0.1cm]
&(1)\ e_{(k)}e_{(i)}=e_{(i)}e_{(k)}=N^{i}e_{(k)}&&\mbox{for $1\leq i\leq k$,}\\[0.1cm]
&(2)\ e_{(i)}s_{2j}e_{(k)}=e_{(k)}s_{2j}e_{(i)}=N^{i-1}e_{(k)}&&\mbox{for $1\leq j\leq i\leq k$,}\\[0.1cm]
&(3)\ s_{2i+1}e_{(k)}=e_{(k)}s_{2i+1}=e_{(k)}&&\mbox{for $0\leq i< k$,}\\[0.1cm]
&(4)\ s_{i}e_{(k)}=e_{(k)}s_{i}&&\mbox{for $i\geq 2k+1$,}\\[0.1cm]
&(5)\ s_{2i-1}s_{2i}e_{(k)}=s_{2i+1}s_{2i}e_{(k)}&&\mbox{for $1\leq i< k$,}\\[0.1cm]
&(6)\ e_{(k)}s_{2i}s_{2i-1}=e_{(k)}s_{2i}s_{2i+1}&&\mbox{for $1\leq i< k$,}\\[0.1cm]
&(7)\ e_{(k+1)}=e_{(1)}s_{2}\cdots s_{2k+1}s_{1}\cdots s_{2k}e_{(k)}&&\mbox{for $1\leq k< \lfloor\frac{n}{2}\rfloor$,}\\[0.1cm]
\end{aligned}
\end{equation*}
\end{definition}

Observe that the subalgebra of $D_{n}(N)$ generated by $s_1,\ldots, s_{n-1}$ is isomorphic to $\mathbb{Z}\mathfrak{S}_n$, where $\mathfrak{S}_n$ is the symmetric group on $n$ letters. It is spanned by the basis diagrams which only have vertical edges. In \cite[\S 2]{Br37} Brauer points out that each basis diagram in $D_n(N)$ which has exactly $2k$ horizontal edges can be obtained in the form $w_1e_{(k)}w_2$, where $w_1$ and $w_2$ are two permutations in $S_n$ and $e_{(k)}$ is a diagram of the following form:
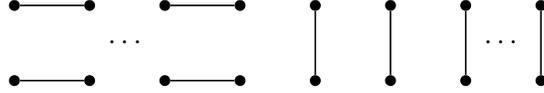
\begin{equation*}
\begin{tikzpicture}[scale=1,semithick]
\node (1) [circle,fill=black,scale=0.4] at (1,1){};
\node (2) [circle,fill=black,scale=0.4] at (2,1){};
\node (2.5) at (2.5,0.5){$\cdots$};
\node (3) [circle,fill=black,scale=0.4] at (3,1){};
\node (4) [circle,fill=black,scale=0.4] at (4,1){};
\node (5) [circle,fill=black,scale=0.4] at (5,1){};
\node (6) [circle,fill=black,scale=0.4] at (6,1){};
\node (7) [circle,fill=black,scale=0.4] at (7,1){};
\node (8) [circle,fill=black,scale=0.4] at (8,1){};
\node (9) [circle,fill=black,scale=0.4] at (1,0){};
\node (10) [circle,fill=black,scale=0.4] at (2,0){};
\node (11) [circle,fill=black,scale=0.4] at (3,0){};
\node (12) [circle,fill=black,scale=0.4] at (4,0){};
\node (13) [circle,fill=black,scale=0.4] at (5,0){};
\node (14) [circle,fill=black,scale=0.4] at (6,0){};
\node (15) [circle,fill=black,scale=0.4] at (7,0){};
\node (16) [circle,fill=black,scale=0.4] at (8,0){};
\node (17) at (7.5,0.5){$\cdots$};

\path (1) edge (2)
          (3) edge (4)
          (5) edge (13)
          (6) edge (14)
          (7) edge (15)
          (8) edge (16)
          (9) edge (10)
          (11) edge (12);
\end{tikzpicture}
\end{equation*}
where each row has exactly $k$ horizontal edges.

Let $s_1,\ldots, s_{n-1}$ be the simple reflections in $\mathfrak{S}_n$ as above. For each $w\in \mathfrak{S}_n$, let $\ell(w)$ be the smallest integer $r\in \Z_{\geqslant 0}$ such that $w=s_{i_1}s_{i_2}\cdots s_{i_r}$; we then say that $s_{i_1}s_{i_2}\cdots s_{i_r}$ is a {\em reduced expression} of $w$ and $\ell(w)$ is the length of $w$. By using the length function on $\mathfrak{S}_n$, Wenzl \cite[\S1.4]{We12a} defined a length function on $D_n(N)$ as follows: for each basis diagram $d\in D_n(N)$ with exactly $2k$ horizontal edges, the length $\ell(d)$ of it is defined by
$$\ell(d)=\hbox{min}\{\ell(\omega_1)+\ell(\omega_2)\:|\:d=\omega_1e_{(k)}\omega_2, \omega_1, \omega_2\in \mathfrak{S}_n\}.$$

For $1\leq i,j\leq n-1$, let
\begin{align*}
s_{i,j}=\begin{cases}s_is_{i+1}\cdots s_{j}& \hbox {if } i\leq j, \\s_is_{i-1}\cdots s_{j}& \hbox {if } i>j.\end{cases}
\end{align*}
It is easy to prove that
\begin{align*}
\mathfrak{S}_n=\mathfrak{S}_{n-1}\bigsqcup(\bigsqcup_{r=1}^{n-1}s_{r,n-1}\mathfrak{S}_{n-1})\hbox{  ~~~~   (a disjoint union),}
\end{align*}
and moreover, $\ell(s_{r,n-1}w)=\ell(s_{r,n-1})+\ell(w)$ for any $w\in \mathfrak{S}_{n-1}$. Hence, we see that for any $w\in \mathfrak{S}_n$, there exist unique elements $t_{n-1}, t_{n-2},\ldots, t_1$ such that $w=t_{n-1}t_{n-2}\cdots t_1$, where $t_j=1$ or $t_j=s_{i_j, j}$ with $1\leq j\leq n-1$ and $1\leq i_j\leq j$, and moreover, $\ell(w)=\ell(t_{n-1})+\ell(t_{n-2})+\cdots+\ell(t_1)$. For each $0\leq k\leq \lfloor\frac{n}{2}\rfloor$, we set
\begin{align}\label{b_k8}
\begin{split}
B_{k}^{*}=\{t_{n-1}t_{n-2}\cdots t_{2k}t_{2k-2}t_{2k-4}\cdots t_2\:|\:\forall ~j,~ t_j=1 \mbox{ or } t_j=s_{i_j, j} \mbox{ for some }1\leq i_j\leq j\},
\end{split}
\end{align}
where $B_{0}^{*}$ is understood as the entire symmetric group $\mathfrak{S}_n$. We set $B_{k} :=\{\omega^{-1}\:|\:\omega\in B_{k}^{*}\}$. Observe that $B_k^*$ has $\frac{n!}{2^k k!}$ elements (cf. \cite[Remark 2.1(3)]{N14}).

\subsection{$q$-Brauer algebras}
Let $q$ and $z$ be two invertible indeterminates.
\begin{definition}{\rm (cf. \cite[Definition 3.1]{We12a}, \cite[Definition 3.1]{N14}) }\label{def:qB1}
Fix $n\in \Z_{\geqslant 2}$. We define the $q$-Brauer algebra $\B_n(q,z)$ over $\Q(q,z)$ with generators $H_1,\ldots, H_{n-1},\Qy$ and the following relations:
\begin{equation*}
\begin{aligned}
&(Q1)\ (H_i-q)(H_i+q^{-1})=0,\\
&(Q2)\ H_iH_{i+1}H_i=H_{i+1}H_iH_{i+1},\\ 
&(Q3)\ H_iH_j=H_jH_i\ \ \text{for }|i-j|>1,\\
&(Q4)\ \Qy^2=\frac{z-z^{-1}}{q-q^{-1}}\Qy,\\
&(Q5)\ H_{1}\Qy=\Qy H_{1}=q\Qy,\\
&(Q6)\ \Qy H_{2}\Qy=z\Qy,\\
&(Q7)\ H_i\Qy=\Qy H_i\ \ \text{for }i>2,\\
&(Q8)\ H_{2}H_{3}H_{1}^{-1}H_{2}^{-1}\Qy_{(2)}=\Qy_{(2)}=\Qy_{(2)}H_{2}H_{3}H_{1}^{-1}H_{2}^{-1}, \text{ where } \Qy_{(2)}=\Qy(H_{2}H_{3}H_{1}^{-1}H_{2}^{-1})\Qy.
\end{aligned}
\end{equation*}
\end{definition}



The following proposition gives the dimension of the $q$-Brauer algebra $\B_n(q,z)$.
\begin{proposition}{\rm (\cite[Theorem 3.8]{We12a})}
The $q$-Brauer algebra $\B_n(q,z)$ is a free $\Q(q,z)$-module of rank $(2n-1)!!=(2n-1)(2n-3)\cdots 1$.
\end{proposition}

Let
\begin{align*}
H_{l,r}^{+}=\begin{cases}H_lH_{l+1}\cdots H_{r}& \hbox {if } l\leq r, \\H_lH_{l-1}\cdots H_{r}& \hbox {if } l>r,\end{cases}
\end{align*}
and
\begin{align*}
H_{l,r}^{-}=\begin{cases}H_l^{-1}H_{l+1}^{-1}\cdots H_{r}^{-1}& \hbox {if } l\leq r, \\H_l^{-1}H_{l-1}^{-1}\cdots H_{r}^{-1}& \hbox {if } l>r,\end{cases}
\end{align*}
for $1\leq l,r\leq n$.

We make the convention that $e_{(0)}=1$. For each $1\leq k\leq \lfloor\frac{n}{2}\rfloor$, we define the elements $e_{(k)}$ in $\B_n(q,z)$ inductively by
$$e_{(1)}=e\quad \mathrm{and}\quad e_{(k+1)}=eH_{2,2k+1}^{+}H_{1,2k}^{-}e_{(k)}~\mathrm{ for }~k\geq 1.$$

\begin{remark} We will abuse the notation by denoting $e_{(k)}$ both a basis diagram in the Brauer algebra $D_{n}(N)$ and an element in the $q$-Brauer algebra $\B_n(q,z)$.
\end{remark}

In the following lemma we shall collect a few identities in $\B_n(q,z)$ which will be used in the sequel.
\begin{lemma}{\rm (cf. \cite[Lemmas 3.2-3.3]{We12a}, \cite[Lemmas 3.3-3.4]{N14}, \cite[Lemma 3.1]{N18}, \cite[Remark 3.10(1)]{N14})}\label{identiesti}
In $\B_n(q,z)$ we have 

$(1)$ $e_{(2)}=e(H_{2}^{-1}H_{1}^{-1}H_{3}H_{2})e=e(H_{2}^{-1}H_{3}^{-1}H_{1}H_{2})e$,

$(2)$ $H_{2j+1}e_{(k)}=e_{(k)}H_{2j+1}=qe_{(k)}$ and $H_{2j+1}^{-1}e_{(k)}=e_{(k)}H_{2j+1}^{-1}=q^{-1}e_{(k)}$ for $0\leq j< k$,

$(3)$ $e_{(k)}H_{2j}H_{2j-1}=e_{(k)}H_{2j}H_{2j+1}$ and $e_{(k)}H_{2j}^{-1}H_{2j-1}^{-1}=e_{(k)}H_{2j}^{-1}H_{2j+1}^{-1}$ for $1\leq j< k$,

$(4)$ $H_{2j-1}H_{2j}e_{(k)}=H_{2j+1}H_{2j}e_{(k)}$ and $H_{2j-1}^{-1}H_{2j}^{-1}e_{(k)}=H_{2j+1}^{-1}H_{2j}^{-1}e_{(k)}$ for $1\leq j< k$,

$(5)$ $\Big(\frac{z-z^{-1}}{q-q^{-1}}\Big)^{j-1}e_{(k+1)}=e_{(j)}H_{2j,2k+1}^{+}H_{2j-1,2k}^{-}e_{(k)}$ for $1\leq j< k$,

$(6)$ $e_{(j)}e_{(k)}=e_{(k)}e_{(j)}=\Big(\frac{z-z^{-1}}{q-q^{-1}}\Big)^{j}e_{(k)}$ for $1\leq j\leq k$,

$(7)$ $e_{(j)}H_{2j}e_{(k)}=e_{(k)}H_{2j}e_{(j)}=z\Big(\frac{z-z^{-1}}{q-q^{-1}}\Big)^{j-1}e_{(k)}$ for $1\leq j\leq k$,

$(8)$ $e_{(i)}H_{j}=H_{j}e_{(i)}$ for $i\geq 1$ and $j\geq 2i+1$.
\end{lemma}

Let $w\in \mathfrak{S}_n$ and let $w=s_{i_1}\cdots s_{i_r}$ be a reduced expression of $w$. It is well-known that the element $H_{w} :=H_{i_1}\cdots H_{i_r}$ does not depend on the choice of the reduced expression of $w$. Let $\mathfrak{S}_{2k+1, n}$ be the subgroup of $\mathfrak{S}_n$ generated by elements $s_{2k+1},s_{2k+2},\ldots, s_{n-1}$. In \cite[\S 3.2]{N14} it has been show that each basis diagram $d\in D_{n}(N)$ with exactly $2k$ horizontal edges can be uniquely represented by a triple $(\omega_1,\omega_{(d)},\omega_2)$ with $\omega_1\in B_{k}^{*}, \omega_2\in B_{k}$ and $\omega_{(d)}\in \mathfrak{S}_{2k+1, n}$ such that $N^{k}d=\omega_1e_{(k)}\omega_{(d)}e_{(k)}\omega_2$ and $\ell(d)=\ell(\omega_1)+\ell(\omega_{(d)})+\ell(\omega_2)$. We call such a unique triple a {\em reduced expression} of $d$.

\begin{definition}{\rm (\cite[Definition 3.12]{N14})}
For each diagram $d$ of $D_{n}(N)$, we define a corresponding element $H_{d}$ in $\B_n(q,z)$ as follows: if $d$ has exactly $2k$ horizontal edges and $(\omega_1,\omega_{(d)},\omega_2)$ is a reduced expression of $d$ with $\omega_1\in B_{k}^{*}, \omega_2\in B_{k}$ and $\omega_{(d)}\in \mathfrak{S}_{2k+1, n}$, then we define $$H_d :=H_{\omega_1}e_{(k)}H_{\omega_{(d)}}H_{\omega_2}.$$ If the diagram $d$ has no horizontal edge, then $d$ is regarded as a permutation $\omega_{(d)}$ of $\mathfrak{S}_n$, and in this case, we define $H_d :=H_{\omega_{(d)}}$.
\end{definition}

Let $I_n$ denote the set of all basis diagrams of the Brauer algebra $D_{n}(N)$. The next proposition gives a standard basis of $\B_n(q,z)$ that is labeled by the basis diagrams of $D_{n}(N)$, which can be used to define a cellular structure on $\B_n(q,z)$.


\begin{proposition}{ \rm(\cite[Theorem 3.13]{N14})}\label{thmhdstandard}
The set $\{H_{d}\:|\:d\in I_n\}$ forms a basis of $\B_n(q,z)$ over $\Q(q,z)$.
\end{proposition}

Let $\mathcal{D}_{k, n}^{*}$ be the set of all diagrams $d^*$ satisfying the following three properties:

(1) $d^*$ has exactly $k$ horizontal edges on each row,

(2) the bottom row of $d^*$ is the same as that of $e_{(k)}$,

(3) there is no crossing between any two vertical edges of $d^*$.
	
Set
\begin{align*}
B_{k, n}^{*} :=\{\omega\in B_{k}^{*}\:|\:d^{*}=\omega e_{(k)}\in \mathcal{D}_{k, n}^{*}\hbox{ and }\ell(d^{*})=\ell(\omega)\},
\end{align*}
and
\begin{align*}
B_{k, n} :=\{\omega^{-1}\:|\:\omega\in B_{k, n}^{*}\}.
\end{align*}
Note that $B_{k, n}^{*}$ has $\frac{n!}{2^k(n-2k)!k!}$ elements (cf. \cite[Remark 3.18(1)]{N14}).

The following lemma gives a decomposition of each element in $B_{k}^{*}$ in terms of $B_{k, n}^{*}$ and $\mathfrak{S}_{2k+1, n}$.
\begin{lemma}{\rm (\cite[Corollary 4.3 and Lemma 4.4]{N14})}\label{permutation unique}
Let $\sigma$ be a permutation of $B_{k}^{*}$. Then there exist unique elements $\omega'\in B_{k, n}^{*}$ and $\pi'\in \mathfrak{S}_{2k+1, n}$ such that $\sigma=\omega'\pi'$ and $\ell(\sigma)=\ell(\omega')+\ell(\pi')$. Similarly, for each element $\varrho\in B_{k}$, there exist unique elements $\tau'\in \mathfrak{S}_{2k+1, n}$ and $\varpi'\in B_{k, n}$ such that $\varrho=\tau'\varpi'$ and $\ell(\varrho)=\ell(\tau')+\ell(\varpi')$.
\end{lemma}

For each $0\leq k\leq \lfloor\frac{n}{2}\rfloor$, let $I_{k,n}$ denote the set of all diagrams in $I_n$ which has exactly $k$ horizontal edges both on the top and bottom rows. By \cite[Lemmas 4.1 and 4.7]{N14}, we obtain the following result.
\begin{lemma}\label{bijection}
There exists a bijection 
$\rho: B_{k, n}^{*}\times \mathfrak{S}_{2k+1, n}\times B_{k, n}\to I_{k,n}$. Under this bijection, if $(\omega_1,\omega_{(d)},\omega_2)\in B_{k, n}^{*}\times \mathfrak{S}_{2k+1, n}\times B_{k, n}$ and $d\in I_{k,n}$ are such that $\rho((\omega_1,\omega_{(d)},\omega_2))=d$, then we have $H_{\omega_1}e_{(k)}H_{\omega_{(d)}}H_{\omega_2}=H_d$, and moreover, $\ell(d)=\ell(\omega_1)+\ell(\omega_{(d)})+\ell(\omega_2)$.
\end{lemma}

\section{A bar involution and canonical bases}
\label{sec:bar}
\subsection{A bar involution}
The following lemma provides an involutive anti-automorphism on $\B_n(q,z)$, which is necessary for establishing its cellularity.

\begin{lemma}{\rm (\cite[Proposition 3.14]{N14})}\label{involutive anti-automorphism}
The map $\jmath$ which is defined by
$$\jmath(e)=e\quad \mathrm{and}\quad\jmath(H_w)=H_{w^{-1}}~\mathrm{ for~each }~w\in \mathfrak{S}_n$$
can be uniquely extended to a $\Q(q,z)$-linear involutive anti-automorphism on $\B_n(q,z)$. Moreover, it satisfies that $\jmath(e_{(k)})=e_{(k)}$ for each $k$.
\end{lemma}

The following lemma provides an involutive automorphism $\overline{\cdot}$, called the bar involution, on $\B_n(q,z)$, which is necessary for constructing its canonical basis.
\begin{lemma}\label{abcdinvo20}
There is a unique involutive homomorphism $\overline{\cdot}$ on $\B_n(q,z)$ which is $\Q$-linear and satisfies $\overline q=q^{-1}, \overline{z}=z^{-1}, \overline{H_{i}}=H_{i}^{-1}$ and $\overline{e}=e$.
\end{lemma}
\begin{proof}
It is easy to check that the homomorphism $\overline{\cdot}$ preserves the relations except (Q8) in Definition \ref{def:qB1}. Thus, if suffices to prove that
\begin{align}\label{2-1}e(H_{2}^{-1}H_{3}^{-1}H_{1}H_{2})e=e(H_{2}^{-1}H_{3}^{-1}H_{1}H_{2})eH_{2}^{-1}H_{3}^{-1}H_{1}H_{2},\end{align}
and
$$e(H_{2}^{-1}H_{3}^{-1}H_{1}H_{2})e=H_{2}^{-1}H_{3}^{-1}H_{1}H_{2}e(H_{2}^{-1}H_{3}^{-1}H_{1}H_{2})e.$$
We only prove \eqref{2-1}, and the second one can be proved similarly.

Since $H_{2}H_{3}H_{1}^{-1}H_{2}^{-1}\Qy_{(2)}=\Qy_{(2)}$, by Lemma \ref{involutive anti-automorphism} we have
$$e(H_{2}^{-1}H_{1}^{-1}H_{3}H_{2})e=e(H_{2}^{-1}H_{1}^{-1}H_{3}H_{2})eH_{2}^{-1}H_{1}^{-1}H_{3}H_{2}.$$
By Lemma \ref{identiesti}(1), we have $$e_{(2)}=e(H_{2}^{-1}H_{1}^{-1}H_{3}H_{2})e=e(H_{2}^{-1}H_{3}^{-1}H_{1}H_{2})e.$$ In order to prove \eqref{2-1}, it suffices to show that
\begin{align}\label{2-2}e(H_{2}^{-1}H_{3}^{-1}H_{1}H_{2})eH_{2}^{-1}H_{3}^{-1}H_{1}H_{2}=e(H_{2}^{-1}H_{3}^{-1}H_{1}H_{2})eH_{2}^{-1}H_{1}^{-1}H_{3}H_{2}.
\end{align}
We have
\begin{align*}
&e(H_{2}^{-1}H_{3}^{-1}H_{1}H_{2})eH_{2}^{-1}H_{3}^{-1}H_{1}H_{2}\\
=&e_{(2)}H_{2}^{-1}(H_3+(q^{-1}-q))(H_1^{-1}+(q-q^{-1}))H_{2}\\
=&e_{(2)}H_{2}^{-1}(H_1^{-1}H_3+(q-q^{-1})H_3+(q^{-1}-q)H_1^{-1}+(q^{-1}-q)(q-q^{-1}))H_2\\
=&e_{(2)}H_{2}^{-1}H_1^{-1}H_3H_2+(q-q^{-1})e_{(2)}H_{2}^{-1}H_3H_2+
\\&\hspace{5cm}(q^{-1}-q)e_{(2)}H_{2}^{-1}H_1^{-1}H_2-(q-q^{-1})^{2}e_{(2)}\\
=&e_{(2)}H_{2}^{-1}H_1^{-1}H_3H_2+(q-q^{-1})e_{(2)}H_3H_2H_{3}^{-1}+(q^{-1}-q)e_{(2)}H_1H_{2}^{-1}H_1^{-1}
\\&\hspace{5cm}-(q-q^{-1})^{2}e_{(2)}.
\end{align*}
Therefore, in order to prove \eqref{2-2} it suffices to show that
\begin{align}\label{eabc}
(q-q^{-1})e_{(2)}H_3H_2H_{3}^{-1}+(q^{-1}-q)e_{(2)}H_1H_{2}^{-1}H_1^{-1}-(q-q^{-1})^{2}e_{(2)}=0.
\end{align}

By Lemma \ref{identiesti}(2), we have $e_{(2)}H_3=e_{(2)}H_1=qe_{(2)}$. By Lemma \ref{identiesti}(3), we have $e_{(2)}H_2H_{3}=e_{(2)}H_2H_1$. Therefore we have
\begin{align*}
&(q-q^{-1})e_{(2)}H_3H_2H_{3}^{-1}+(q^{-1}-q)e_{(2)}H_1H_{2}^{-1}H_1^{-1}-(q-q^{-1})^{2}e_{(2)}\\
=&(q^{2}-1)e_{(2)}H_2H_{3}^{-1}+(1-q^{2})e_{(2)}H_{2}^{-1}H_1^{-1}-(q-q^{-1})^{2}e_{(2)}\\
=&(q^{2}-1)e_{(2)}H_2(H_{3}+(q^{-1}-q))+(1-q^{2})e_{(2)}(H_2+(q^{-1}-q))(H_1+(q^{-1}-q))\\
&-(q-q^{-1})^{2}e_{(2)}\\
=&(q^{2}-1)e_{(2)}H_2H_{3}-q^{-1}(1-q^{2})^{2}e_{(2)}H_2+(1-q^{2})e_{(2)}H_2H_1+q^{-1}(1-q^{2})^{2}e_{(2)}H_2\\
&+(1-q^{2})(q^{-1}-q)e_{(2)}H_1+(1-q^{2})(q^{-1}-q)^{2}e_{(2)}-(q-q^{-1})^{2}e_{(2)}\\
=&0.
\end{align*}
Thus, \eqref{eabc} holds and we are done.
\end{proof}

\begin{lemma}
For each $1\leq k\leq \lfloor\frac{n}{2}\rfloor$, we have $\overline{e_{(k)}}=e_{(k)}$.
\end{lemma}
\begin{proof}
We first prove that
\begin{align}\label{ek+1}
e_{(l+1)}=\Big(\frac{q-q^{-1}}{z-z^{-1}}\Big)^{l-1}e_{(l)}H_{2l}H_{2l+1}H_{2l-1}^{-1}H_{2l}^{-1}e_{(l)}
\end{align}
for $l\geq 1$. We shall prove \eqref{ek+1} by induction on $l$. When $l=1$, \eqref{ek+1} holds by definition. We assume \eqref{ek+1} holds for $l-1$. By Lemma \ref{identiesti}(5) we have
\begin{align}\label{ek+2}
e_{(l+1)}=\Big(\frac{q-q^{-1}}{z-z^{-1}}\Big)^{l-2}e_{(l-1)}H_{2l-2}H_{2l-1}H_{2l}H_{2l+1}H_{2l-3}^{-1}H_{2l-2}^{-1}H_{2l-1}^{-1}H_{2l}^{-1}e_{(l)}.
\end{align}
By Lemma \ref{identiesti}(8) we have $e_{(i)}H_{j}=H_{j}e_{(i)}$ for $j\geq 2i+1$. Moreover, by Lemma \ref{identiesti}(6) we have
$e_{(l)}=\big(\frac{q-q^{-1}}{z-z^{-1}}\big)^{l-1}e_{(l-1)}e_{(l)}$. Therefore, by \eqref{ek+2} and the assumption that \eqref{ek+1} holds for $l-1$, we have
\begin{align*}
&e_{(l+1)}\\
=&\Big(\frac{q-q^{-1}}{z-z^{-1}}\Big)^{l-1}\Big(\frac{q-q^{-1}}{z-z^{-1}}\Big)^{l-2}e_{(l-1)}H_{2l-2}
H_{2l-1}H_{2l-3}^{-1}H_{2l-2}^{-1}e_{(l-1)}\\
&\times H_{2l}H_{2l+1}H_{2l-1}^{-1}H_{2l}^{-1}e_{(l)}\\
=&\Big(\frac{q-q^{-1}}{z-z^{-1}}\Big)^{l-1}e_{(l)}H_{2l}H_{2l+1}H_{2l-1}^{-1}H_{2l}^{-1}e_{(l)}.
\end{align*}

Next we prove the lemma by induction on $k$. By Lemma \ref{abcdinvo20} and Lemma \ref{identiesti}(1), we have $\overline{e_{(1)}}=e_{(1)}$and $\overline{e_{(2)}}=e_{(2)}$, that is, the lemma holds for $k=1, 2$. We assume that it is true for $k$ and want to show that $\overline{e_{(k+1)}}=e_{(k+1)}$. By \eqref{ek+1} we have
\begin{align*}
\overline{e_{(k+1)}}=\Big(\frac{q-q^{-1}}{z-z^{-1}}\Big)^{k-1}e_{(k)}H_{2k}^{-1}H_{2k+1}^{-1}H_{2k-1}H_{2k}e_{(k)}.
\end{align*}
By Lemma \ref{involutive anti-automorphism} and \eqref{ek+1} we have
\begin{align*}
e_{(k+1)}=\jmath(e_{(k+1)})=\Big(\frac{q-q^{-1}}{z-z^{-1}}\Big)^{k-1}e_{(k)}H_{2k}^{-1}H_{2k-1}^{-1}H_{2k+1}H_{2k}e_{(k)}.
\end{align*}
Therefore, in order to prove that $\overline{e_{(k+1)}}=e_{(k+1)}$, it suffices to show that 
\begin{align}\label{habc}
e_{(k)}H_{2k}^{-1}H_{2k+1}^{-1}H_{2k-1}H_{2k}e_{(k)}=e_{(k)}H_{2k}^{-1}H_{2k+1}H_{2k-1}^{-1}H_{2k}e_{(k)}.
\end{align}
We have
\begin{align*}
&e_{(k)}H_{2k}^{-1}H_{2k+1}^{-1}H_{2k-1}H_{2k}e_{(k)}\\
=&e_{(k)}H_{2k}^{-1}(H_{2k+1}+(q^{-1}-q))(H_{2k-1}^{-1}+(q-q^{-1}))H_{2k}e_{(k)}\\
=&e_{(k)}H_{2k}^{-1}H_{2k+1}H_{2k-1}^{-1}H_{2k}e_{(k)}+(q-q^{-1})e_{(k)}H_{2k}^{-1}H_{2k+1}H_{2k}e_{(k)}\\
&+(q^{-1}-q))e_{(k)}H_{2k}^{-1}H_{2k-1}^{-1}H_{2k}e_{(k)}-(q-q^{-1})^{2}e_{(k)}^{2}\\
=&e_{(k)}H_{2k}^{-1}H_{2k+1}H_{2k-1}^{-1}H_{2k}e_{(k)}+(q-q^{-1})e_{(k)}H_{2k+1}H_{2k}H_{2k+1}^{-1}e_{(k)}\\
&+(q^{-1}-q))e_{(k)}H_{2k-1}H_{2k}^{-1}H_{2k-1}^{-1}e_{(k)}-(q-q^{-1})^{2}e_{(k)}^{2}.
\end{align*}
By Lemma \ref{identiesti}(7) and (8), we have
\begin{align*}
e_{(k)}H_{2k+1}H_{2k}H_{2k+1}^{-1}e_{(k)}=&H_{2k+1}e_{(k)}H_{2k}e_{(k)}H_{2k+1}^{-1}\\
=&z\Big(\frac{z-z^{-1}}{q-q^{-1}}\Big)^{k-1}H_{2k+1}e_{(k)}H_{2k+1}^{-1}\\
=&e_{(k)}H_{2k}e_{(k)}.
\end{align*}
By Lemma \ref{identiesti}(2), we have $$e_{(k)}H_{2k-1}H_{2k}^{-1}H_{2k-1}^{-1}e_{(k)}=e_{(k)}H_{2k}^{-1}e_{(k)}.$$ Therefore, we have
\begin{align*}
&e_{(k)}H_{2k}^{-1}H_{2k+1}^{-1}H_{2k-1}H_{2k}e_{(k)}\\
=&e_{(k)}H_{2k}^{-1}H_{2k+1}H_{2k-1}^{-1}H_{2k}e_{(k)}+(q-q^{-1})e_{(k)}(H_{2k}-H_{2k}^{-1})e_{(k)}-(q-q^{-1})^{2}e_{(k)}^{2}\\
=&e_{(k)}H_{2k}^{-1}H_{2k+1}H_{2k-1}^{-1}H_{2k}e_{(k)}.
\end{align*}
Thus, \eqref{habc} holds and we are done.
\end{proof}

\subsection{Canonical bases}
In this subsection we shall construct a Kazhdan-Lusztig-type basis on the $q$-Brauer algebra $\B_n(q,z)$.
\begin{lemma}{\em(\cite[Lemma 1.2(a)]{We12a})}\label{B_ke_k}
For any $w\in \mathfrak{S}_n$ and $1\leq k\leq \lfloor\frac{n}{2}\rfloor$, there exists a unique element $\sigma\in B_{k}^{*}$ such that $we_{(k)}=\sigma e_{(k)}$ and $\ell(\sigma e_{(k)})=\ell(\sigma)\leq \ell(w)$.
\end{lemma}

In fact, the element $\sigma\in B_{k}^{*}$ in Lemma \ref{B_ke_k} can be constructed as follows (refer to the proof of \cite[Lemma 1.2(a)]{We12a}). We set $d=we_{(k)}$. Using exactly the same arguments as those before \eqref{b_k8}, we see that there exist unique elements $t_{n-1}, t_{n-2},\ldots, t_{2k}$ such that $d'=(t_{n-1}t_{n-2}\cdots t_{2k})^{-1}d$ is a diagram in $\mathfrak{S}_{2k}e_{(k)}$, and moreover, $\ell(t_{n-1}t_{n-2}\cdots t_{2k}y)=\ell(t_{n-1})+\ell(t_{n-2})+\cdots +\ell(t_{2k})+\ell(y)$ for any $y\in \mathfrak{S}_{2k}$. Let $i_{2k-2}$ be the label of the vertex of $d'$ which is connected with the $2k$-th vertex on the top row. If $i_{2k-2}=2k-1$, we set $t_{2k-2}=1$; if $i_{2k-2}\leq 2k-2$, then we set $t_{2k-2}=s_{i_{2k-2}, 2k-2}$. Then in the diagram $d''=t_{2k-2}^{-1}d'$, the $(2k-1)$-st and $2k$-th vertices on the top row are connected by a horizontal edge. Proceeding in this way, we see that there exist some elements $t_{2k-2}, t_{2k-4},\ldots,t_{2}$ such that $e_{(k)}=t_{2}^{-1}\cdots t_{2k-4}^{-1}t_{2k-2}^{-1}d'$, that is, $d'=t_{2k-2}t_{2k-4}\cdots t_{2}e_{(k)}$. Set $$\sigma :=t_{n-1}t_{n-2}\cdots t_{2k}t_{2k-2}t_{2k-4}\cdots t_{2}.$$ Then $\sigma$ is just the required element in Lemma \ref{B_ke_k}, that is, $\sigma\in B_{k}^{*}$ is such that $we_{(k)}=\sigma e_{(k)}$ and $\ell(\sigma e_{(k)})=\ell(\sigma)\leq \ell(w)$. From the above process, we see that the choices of the elements $t_{2k-2},t_{2k-4},\ldots,t_{2}$ depend only on the defining relations $(S_1)$-$(S_3)$, $(3)$ and $(5)$ in Definition \ref{brau} (refer to the last paragraph on \cite[p. 1385]{N14}).

In an analogous way, by using the corresponding relations $(Q1)$-$(Q3)$ on the generators $H_i$ in Definition \ref{def:qB1} as well as two relations (2) and (4) in Lemma \ref{identiesti}, we see that the element $H_{w}e_{(k)}$ transforms into the form $$\sum\limits_{\substack{\sigma'\in B_{k}^{*}\\\ell(\sigma')\leq\ell(w)}}r_{\sigma', w}H_{\sigma'}e_{(k)}$$
for some $r_{\sigma', w}\in \mathbb{Z}[q, q^{-1}]$ (refer to the proof of \cite[Lemma 4.10]{N14}).

Let us look at an example.
\begin{example}
Fix $n=7$ and $k=3$. Assume that $w=s_{6}s_{1, 5}s_{2, 4}s_{2}\in \mathfrak{S}_7$.

We set $t_{6}=s_{6}$. Then $t_{6}^{-1}we_{(3)}\in \mathfrak{S}_{6}e_{(3)}$ and
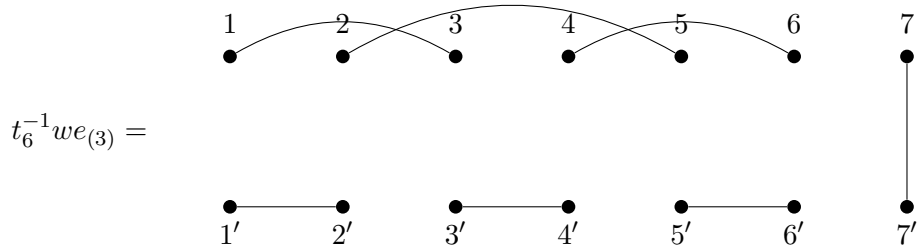
\[
\begin{tikzpicture}
    \node (0) at (-2,0) {$t_{6}^{-1}we_{(3)}=$};
    \node (u1) [fill,circle,label=above:{$1$},scale=0.5] at (0,1){};
    \node (u2) [fill,circle,label=above:{$2$},scale=0.5] at (1.5,1){};
    \node (u3) [fill,circle,label=above:{$3$},scale=0.5] at (3,1){};
    \node (u4) [fill,circle,label=above:{$4$},scale=0.5] at (4.5,1){};
    \node (u5) [fill,circle,label=above:{$5$},scale=0.5] at (6,1){};
    \node (u6) [fill,circle,label=above:{$6$},scale=0.5] at (7.5,1){};
    \node (u7) [fill,circle,label=above:{$7$},scale=0.5] at (9,1){};
     \node (d1) [fill,circle,label=below:{$1'$},scale=0.5] at (0,-1){};
    \node (d2) [fill,circle,label=below:{$2'$},scale=0.5] at (1.5,-1){};
    \node (d3) [fill,circle,label=below:{$3'$},scale=0.5] at (3,-1){};
    \node (d4) [fill,circle,label=below:{$4'$},scale=0.5] at (4.5,-1){};
    \node (d5) [fill,circle,label=below:{$5'$},scale=0.5] at (6,-1){};
    \node (d6) [fill,circle,label=below:{$6'$},scale=0.5] at (7.5,-1){};
    \node (d7) [fill,circle,label=below:{$7'$},scale=0.5] at (9,-1){};
\path
    (u1) edge [bend left] (u3)
    (u2) edge [bend left] (u5)
    (d6) edge (d5)
    (u4) edge [bend left] (u6)
    (d1) edge (d2)
    (d3) edge (d4)
    (u7) edge (d7);
\end{tikzpicture}
\]
In the diagram $t_{6}^{-1}we_{(3)}$, we see that the label of its vertex which is connected with the $6$-th vertex on the top row is $4$. Thus, we set $t_{4}=s_{4, 4}=s_{4}$. Then we have
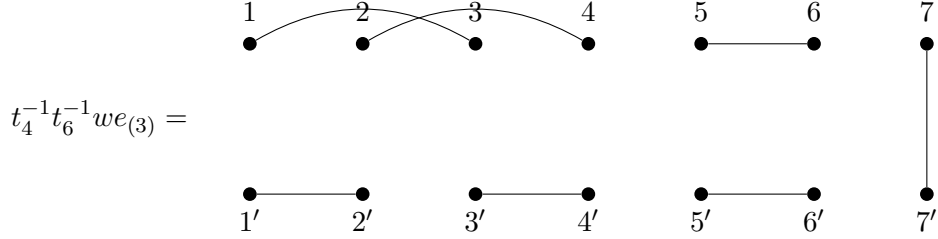
\[
\begin{tikzpicture}
    \node (0) at (-2,0) {$t_{4}^{-1}t_{6}^{-1}we_{(3)}=$};
    \node (u1) [fill,circle,label=above:{$1$},scale=0.5] at (0,1){};
    \node (u2) [fill,circle,label=above:{$2$},scale=0.5] at (1.5,1){};
    \node (u3) [fill,circle,label=above:{$3$},scale=0.5] at (3,1){};
    \node (u4) [fill,circle,label=above:{$4$},scale=0.5] at (4.5,1){};
    \node (u5) [fill,circle,label=above:{$5$},scale=0.5] at (6,1){};
    \node (u6) [fill,circle,label=above:{$6$},scale=0.5] at (7.5,1){};
    \node (u7) [fill,circle,label=above:{$7$},scale=0.5] at (9,1){};
     \node (d1) [fill,circle,label=below:{$1'$},scale=0.5] at (0,-1){};
    \node (d2) [fill,circle,label=below:{$2'$},scale=0.5] at (1.5,-1){};
    \node (d3) [fill,circle,label=below:{$3'$},scale=0.5] at (3,-1){};
    \node (d4) [fill,circle,label=below:{$4'$},scale=0.5] at (4.5,-1){};
    \node (d5) [fill,circle,label=below:{$5'$},scale=0.5] at (6,-1){};
    \node (d6) [fill,circle,label=below:{$6'$},scale=0.5] at (7.5,-1){};
    \node (d7) [fill,circle,label=below:{$7'$},scale=0.5] at (9,-1){};
\path
    (u1) edge [bend left] (u3)
    (u2) edge [bend left] (u4)
    (d6) edge (d5)
    (u5) edge (u6)
    (d1) edge (d2)
    (d3) edge (d4)
    (u7) edge (d7);
\end{tikzpicture}
\]
In the diagram $t_{4}^{-1}t_{6}^{-1}we_{(3)}$, the $5$-th and $6$-th vertices on the top row are connected by a horizontal edge and the label of its vertex which is connected with the $4$-th vertex on the top row is $2$. Thus we set $t_{2}=s_{2, 2}=s_{2}$. Then we have
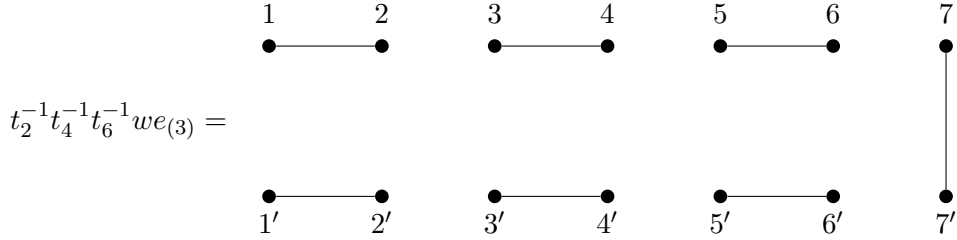
\[
\begin{tikzpicture}
    \node (0) at (-2,0) {$t_{2}^{-1}t_{4}^{-1}t_{6}^{-1}we_{(3)}=$};
    \node (u1) [fill,circle,label=above:{$1$},scale=0.5] at (0,1){};
    \node (u2) [fill,circle,label=above:{$2$},scale=0.5] at (1.5,1){};
    \node (u3) [fill,circle,label=above:{$3$},scale=0.5] at (3,1){};
    \node (u4) [fill,circle,label=above:{$4$},scale=0.5] at (4.5,1){};
    \node (u5) [fill,circle,label=above:{$5$},scale=0.5] at (6,1){};
    \node (u6) [fill,circle,label=above:{$6$},scale=0.5] at (7.5,1){};
    \node (u7) [fill,circle,label=above:{$7$},scale=0.5] at (9,1){};
     \node (d1) [fill,circle,label=below:{$1'$},scale=0.5] at (0,-1){};
    \node (d2) [fill,circle,label=below:{$2'$},scale=0.5] at (1.5,-1){};
    \node (d3) [fill,circle,label=below:{$3'$},scale=0.5] at (3,-1){};
    \node (d4) [fill,circle,label=below:{$4'$},scale=0.5] at (4.5,-1){};
    \node (d5) [fill,circle,label=below:{$5'$},scale=0.5] at (6,-1){};
    \node (d6) [fill,circle,label=below:{$6'$},scale=0.5] at (7.5,-1){};
    \node (d7) [fill,circle,label=below:{$7'$},scale=0.5] at (9,-1){};
\path
    (u1) edge  (u2)
    (u4) edge  (u3)
    (d6) edge (d5)
    (u5) edge (u6)
    (d1) edge (d2)
    (d3) edge (d4)
    (u7) edge (d7);
\end{tikzpicture}
\]
Therefore $t_{2}^{-1}t_{4}^{-1}t_{6}^{-1}we_{(3)}=e_{(3)}$. We set $\sigma=t_{6}t_{4}t_{2}=s_{6}s_{4}s_{2}$. Then, $\sigma\in B_{3}^{*}$ satisfies that $we_{(3)}=\sigma e_{(3)}$ and $\ell(\sigma e_{(3)})=\ell(\sigma)< \ell(w)$.

We can give an equivalent description of the above procedure using relations $(S_1)$-$(S_3)$, $(3)$ and $(5)$ in Definition \ref{brau}. We have
\begin{align*}
we_{(3)}&\overset{(S_3)}{=}s_{6}s_{1, 4}s_{2, 3}s_{2}(s_{5}s_{4}e_{(3)})\overset{(5)}{=}s_{6}s_{1, 4}s_{2, 3}s_{2}(s_{3}s_{4}e_{(3)})\overset{(S_2)}{=}s_{6}s_{1, 4}s_{2}(s_{2}s_{3}s_{2})s_{4}e_{(3)}\\
&\overset{(S_1), (S_3)}{=}s_{6}s_{1, 3}(s_{4}s_{3}s_{4})s_{2}e_{(3)}\overset{(S_1), (S_2)}{=}s_{6}s_{1, 2}s_{4}(s_{3}s_{2}e_{(3)})\overset{(S_3), (5)}{=}s_{6}s_{4}s_{1}s_{2}(s_{1}s_{2}e_{(3)})\\
&\overset{(S_1), (S_2)}{=}s_{6}s_{4}s_{2}(s_{1}e_{(3)})\overset{(3)}{=}s_{6}s_{4}s_{2}e_{(3)}.
\end{align*}
We set $\sigma=s_{6}s_{4}s_{2}$. Then $\sigma\in B_{3}^{*}$ is the desired element.

In an analogous way, in $\B_7(q,z)$ we have
\begin{align*}
&H_{w}e_{(3)}=H_{6}H_{1,5}^{+}H_{2,4}^{+}H_{2}e_{(3)}\overset{(Q3)}{=}H_{6}H_{1,4}^{+}H_{2,3}^{+}H_{2}(H_{5}H_{4}e_{(3)})\\
&\overset{Lemma ~\ref{identiesti}(4)}{=}H_{6}H_{1,4}^{+}H_{2,3}^{+}H_{2}(H_{3}H_{4}e_{(3)})\overset{(Q2)}{=}H_{6}H_{1,4}^{+}H_{2}(H_{2}H_{3}H_{2})H_{4}e_{(3)}\\
&\overset{(Q2), (Q3)}{=}H_{6}H_{1,3}^{+}H_{2}^{2}(H_{3}H_{4}H_{3})H_{2}e_{(3)}\\
&\overset{(Q1), Lemma ~\ref{identiesti}(4)}{=}H_{6}H_{1,3}^{+}((q-q^{-1})H_{2}+1)H_{3}H_{4}(H_{1}H_{2}e_{(3)})\\
&=(q-q^{-1})H_{6}H_{1,2}^{+}(H_{2}H_{3}H_{2})H_{4}H_{1}H_{2}e_{(3)}+H_{6}H_{1,2}^{+}((q-q^{-1})H_{3}+1)H_{4}H_{1}H_{2}e_{(3)}\\
&=(q-q^{-1})H_{6}H_{1}((q-q^{-1})H_{2}+1)H_{3}H_{4}H_{1}H_{2}H_{1}e_{(3)}\\
&\quad+(q-q^{-1})H_{6}H_{1,4}^{+}H_{1,2}^{+}e_{(3)}+H_{6}H_{4}((q-q^{-1})H_{1}+1)H_{2}H_{1}e_{(3)}\\
&\overset{Lemma ~\ref{identiesti}(2)}{=}q(q-q^{-1})^{2}H_{6}H_{1,4}^{+}H_{1,2}^{+}e_{(3)}+q(q-q^{-1})H_{6}H_{3,4}^{+}((q-q^{-1})H_{1}+1)H_{2}e_{(3)}\\
&\qquad+(q-q^{-1})H_{6}H_{1,4}^{+}H_{1,2}^{+}e_{(3)}+q(q-q^{-1})H_{6}H_{4}H_{1,2}^{+}e_{(3)}+qH_{6}H_{4}H_{2}e_{(3)}\\
&=q^{2}(q-q^{-1})H_{6}H_{1,4}^{+}H_{1,2}^{+}e_{(3)}+q(q-q^{-1})^{2}H_{6}H_{3,4}^{+}H_{1,2}^{+}e_{(3)}\\
&\quad+q(q-q^{-1})H_{6}H_{3,4}^{+}H_{2}e_{(3)}+q(q-q^{-1})H_{6}H_{4}H_{1,2}^{+}e_{(3)}+qH_{6}H_{4}H_{2}e_{(3)}.
\end{align*}
Thus, we see that the element $H_{w}e_{(3)}$ can be written as a $\mathbb{Z}[q, q^{-1}]$-linear combination of elements $H_{\sigma_{j}}e_{(3)}$ $(1\leq j\leq 5)$, where each $\sigma_{j}$ satisfies that $\sigma_{j}\in B_{3}^{*}$ and $\ell(\sigma_{j})< \ell(w)$.
\end{example}

Summarizing, we obtain the following lemma.
\begin{lem}\label{H_B_ke_k}
For any $w\in \mathfrak{S}_n$, we have $$H_{w}e_{(k)}=\sum\limits_{\substack{\sigma'\in B_{k}^{*}\\\ell(\sigma')\leq\ell(w)}}r_{\sigma', w}H_{\sigma'}e_{(k)}$$
for some $r_{\sigma', w}\in \mathbb{Z}[q, q^{-1}]$.
\end{lem}

Applying Lemma \ref{involutive anti-automorphism}, we immediately get the following lemma.
\begin{lem}\label{e_kH_B_k}
For any $y\in \mathfrak{S}_n$, we have $$e_{(k)}H_{y}=\sum\limits_{\substack{\varpi'\in B_{k}\\\ell(\varpi')\leq\ell(y)}}s_{\varpi', y}e_{(k)}H_{\varpi'}$$
for some $s_{\varpi', y}\in \mathbb{Z}[q, q^{-1}]$.
\end{lem}

\begin{lem}\label{wyomega}
For each $w, y\in \mathfrak{S}_n$ and $\omega_{(d)}\in \mathfrak{S}_{2k+1, n}$, we have
$$H_{w}e_{(k)}H_{\omega_{(d)}}H_{y}=\sum\limits_{\substack{a\in I_{k,n}\\\ell(a)\leq\ell(w)+\ell(\omega_{(d)})+\ell(y)}}r_{a}H_a$$
for some $r_{a}\in \mathbb{Z}[q, q^{-1}]$.
\end{lem}
\begin{proof}
By Lemma \ref{identiesti}(8) we see that $e_{(k)}H_{w}=H_{w}e_{(k)}$ for any $w\in \mathfrak{S}_{2k+1, n}$. By Lemma \ref{identiesti}(6), we have $e_{(k)}^{2}=\big(\frac{z-z^{-1}}{q-q^{-1}}\big)^{k}e_{(k)}$. Thus, by Lemmas \ref{permutation unique}, \ref{H_B_ke_k} and \ref{e_kH_B_k}, we have
\begin{align*}
H_{w}e_{(k)}H_{\omega_{(d)}}H_{y}=&\Big(\frac{q-q^{-1}}{z-z^{-1}}\Big)^{k}\times
\sum\limits_{\substack{(\omega', \pi')\in B_{k, n}^{*}\times \mathfrak{S}_{2k+1, n}\\\ell(\omega')+\ell(\pi')\leq \ell(w)}}r_{\omega', \pi'}H_{\omega'}H_{\pi'}e_{(k)}\\
&\times H_{\omega_{(d)}}\times
\sum\limits_{\substack{(\tau', \varpi')\in \mathfrak{S}_{2k+1, n}\times B_{k, n} \\\ell(\tau')+\ell(\varpi')\leq \ell(y)}}s_{\tau', \varpi'}e_{(k)}H_{\tau'}H_{\varpi'}
\end{align*}
for some $r_{\omega', \pi'}, s_{\tau', \varpi'}\in \mathbb{Z}[q, q^{-1}]$.

For any $\pi', \tau'\in \mathfrak{S}_{2k+1, n}$ as above, we have $$H_{\pi'}H_{\omega_{(d)}}H_{\tau'}=\sum\limits_{\substack{\chi\in \mathfrak{S}_{2k+1, n}\\\ell(\chi)\leq\ell(\pi')+\ell(\omega_{(d)})+\ell(\tau')}}t_{\pi', \tau'}^{\chi}H_{\chi}$$
for some $t_{\pi', \tau'}^{\chi}\in \mathbb{Z}[q, q^{-1}]$. Thus, we have
\begin{align*}
H_{w}e_{(k)}H_{\omega_{(d)}}H_{y}=\sum\limits_{\substack{(\omega',\chi,\varpi')\in B_{k, n}^{*}\times \mathfrak{S}_{2k+1, n}\times B_{k, n}\\\ell(\omega')+\ell(\chi)+\ell(\varpi')\leq \ell(w)+\ell(\omega_{(d)})+\ell(y)}}r_{\omega'}s_{\varpi'}t^{\chi}H_{\omega'}e_{(k)}H_{\chi}H_{\varpi'}
\end{align*}
for some $r_{\omega'}, s_{\varpi'}, t^{\chi}\in \mathbb{Z}[q, q^{-1}]$.

By Lemma \ref{bijection}, we see that
$$H_{w}e_{(k)}H_{\omega_{(d)}}H_{y}=\sum\limits_{\substack{a\in I_{k,n}\\\ell(a)\leq\ell(w)+\ell(\omega_{(d)})+\ell(y)}}r_{a}H_a$$
for some $r_{a}\in \mathbb{Z}[q, q^{-1}]$.
\end{proof}

\begin{lem}\label{bar_H_d}\label{overlinehd}
For each diagram $d\in  I_{k,n}$, we have
$$\overline{H_{d}}=H_{d}+\sum\limits_{\substack{d'\in  I_{k,n}\\\ell(d')<\ell(d)}}r_{d', d}H_{d'}$$
for some $r_{d', d}\in \mathbb{Z}[q, q^{-1}].$
\end{lem}
\begin{proof}
For $k=0$, that is, $d$ has no horizontal edge, it is well known. Assume that $1\leq k\leq \lfloor\frac{n}{2}\rfloor$. By Lemma \ref{bijection}, if $(\omega_1,\omega_{(d)},\omega_2)\in B_{k, n}^{*}\times \mathfrak{S}_{2k+1, n}\times B_{k, n}$ is such that $\rho((\omega_1,\omega_{(d)},\omega_2))=d$, then we have $H_d=H_{\omega_1}e_{(k)}H_{\omega_{(d)}}H_{\omega_2}$ and $\ell(d)=\ell(\omega_1)+\ell(\omega_{(d)})+\ell(\omega_2).$ We have
\begin{align*}
\overline{H_{d}}=&\overline{H_{\omega_1}}e_{(k)}\overline{H_{\omega_{(d)}}}\overline{H_{\omega_2}}\\
=&\bigg(H_{\omega_1}+\sum\limits_{\omega'_1; \ell(\omega'_1)<\ell(\omega_1)}a_{\omega'_1, \omega_1}H_{\omega'_1}\bigg)e_{(k)} \bigg(H_{\omega_{(d)}}+\sum\limits_{\substack{\omega_{(d')}\in \mathfrak{S}_{2k+1, n}\\\ell(\omega_{(d')})<\ell(\omega_{(d)})}}b_{\omega_{(d')}, \omega_{(d)}}H_{\omega_{(d')}}\bigg)\\
&\times \bigg(H_{\omega_2}+\sum\limits_{\omega'_2; \ell(\omega'_2)<\ell(\omega_2)}c_{\omega'_2, \omega_2}H_{\omega'_2}\bigg).
\end{align*}
By Lemma \ref{wyomega}, we obtain the desired result.
\end{proof}

By Proposition \ref{thmhdstandard}, Lemma \ref{overlinehd} and Lusztig's lemma (cf. \cite[Lemma 24.2.1]{Lu94}), we obtain the canonical and dual canonical basis for $\B_n(q,z)$ over $\Q(q,z)$.
\begin{theorem}\label{thmhdcanonical}
There exists a unique basis $\{C_{d}\:|\:d\in  I_{k,n}, 0\leq k\leq \lfloor\frac{n}{2}\rfloor\}$ of $\B_n(q,z)$ over $\Q(q,z)$, called the canonical basis, such that
\begin{align*}
(1) ~~\overline{C_{d}}&=C_{d},\\
(2) ~~C_{d}&=H_{d}+\sum\limits_{\substack{d'\in I_{k,n} \\ \ell(d')<\ell(d)}}p_{d', d}H_{d'}, ~\mbox{where } p_{d', d}\in q^{-1}\mathbb{Z}[q^{-1}].
\end{align*}
\end{theorem}
\begin{theorem}
There exists a unique basis $\{C^*_{d}\:|\:d\in  I_{k,n}, 0\leq k\leq \lfloor\frac{n}{2}\rfloor\}$ of $\B_n(q,z)$ over $\Q(q,z)$, called the dual canonical basis, such that
\begin{align*}
(1) ~~\overline{C^*_{d}}&=C^*_{d},\\
(2) ~~C^*_{d}&=H_{d}+\sum\limits_{\substack{d'\in I_{k,n} \\ \ell(d')<\ell(d)}}p^*_{d', d}H_{d'}, ~\mbox{where } p^*_{d', d}\in q\mathbb{Z}[q].
\end{align*}
\end{theorem}

\begin{remark}
Note that in the above theorems, the coefficients $p_{d', d}$ (resp. $p^*_{d', d}$) are polynomials in $q^{-1}$ (resp. $q$), which do not depend on $z$ (compare with \cite[\S 5.2]{FG95}).
\end{remark}

\begin{remark}
Fix some $0\leq k\leq \lfloor\frac{n}{2}\rfloor$ and $d\in  I_{k,n}$. Assume that by Lemma \ref{bijection}, $(\omega_1,\omega_{(d)},\omega_2)$ is the unique element in $B_{k, n}^{*}\times \mathfrak{S}_{2k+1, n}\times B_{k, n}$ such that $\rho((\omega_1,\omega_{(d)},\omega_2))=d$. It is clear from the definition that $(\omega_2^{-1},\omega_{(d)}^{-1},\omega_1^{-1})$ also belongs to $B_{k, n}^{*}\times \mathfrak{S}_{2k+1, n}\times B_{k, n}$, and we can assume that $d'\in  I_{k,n}$ is such that $\rho((\omega_2^{-1},\omega_{(d)}^{-1},\omega_1^{-1}))=d'$.

It is easy to check that the bar involution $\overline{\cdot}$ on $\B_n(q,z)$ commutes with the anti-involution $\jmath$, and moreover, $\jmath$ is $\Q(q,z)$-linear. Therefore, we have $\jmath(C_{d})=C_{d'}$ and $\jmath(C^*_{d})=C^*_{d'}$.
\end{remark}

Finally, let us look at some examples.
\begin{example}

(1) When $n=2$, the canonical basis of $\B_2(q,z)$ is given by $\{1,\Qy,H_1+q^{-1}\}$.

(2) When $n=3$, the canonical basis of $\B_3(q,z)$ is given by 
\begin{align*}
&C_{0}=1,\quad C_1=H_1+q^{-1},\quad C_2=H_2+q^{-1}, \\
&C_{12}=H_1H_2+q^{-1}H_1+q^{-1}H_2+q^{-2},\\
&C_{21}=H_2H_1+q^{-1}H_1+q^{-1}H_2+q^{-2},\\
&C_{121}=H_1H_2H_1+q^{-1}H_1H_2+q^{-1}H_2H_1+q^{-2}H_1+q^{-2}H_2+q^{-3},\\
&C_e=\Qy,\quad C_{2e}=H_2\Qy+q^{-1}\Qy,\quad C_{e2}=\Qy H_2+q^{-1}\Qy,\\
&C_{2e2}=H_2\Qy H_2+q^{-1}H_2\Qy+q^{-1}\Qy H_2+q^{-2}\Qy,\\
&C_{12e}=H_1H_2\Qy+q^{-1}H_2\Qy+q^{-2}\Qy,\\
&C_{e21}=\Qy H_2H_1+q^{-1}\Qy H_2+q^{-2}\Qy,\\
&C_{12e2}=H_1H_2\Qy H_2+q^{-1}H_2\Qy H_2+q^{-1}H_1H_2\Qy +q^{-2}H_2\Qy+q^{-2}\Qy H_2+q^{-3}\Qy,\\
&C_{2e21}=H_2\Qy H_2H_1+q^{-1}H_2\Qy H_2+q^{-1}\Qy H_2H_1 +q^{-2}H_2\Qy+q^{-2}\Qy H_2+q^{-3}\Qy,\\
&C_{12e21}=H_1H_2\Qy H_2H_1+q^{-1}H_1H_2\Qy H_2+q^{-1}H_2\Qy H_2H_1+q^{-2}H_1H_2\Qy\atop+q^{-2}H_2\Qy H_2+q^{-2}\Qy H_2H_1+q^{-3}H_2\Qy+q^{-3}\Qy H_2+q^{-4}\Qy.
\end{align*}
Moreover, we can compute the structure constants of $\B_3(q,z)$ with respect to the above basis. For example, we have
\begin{align*}
&C_1 \cdot C_{2e}=C_{12e}+C_e,\\
&C_{2e}\cdot C_{e2}=\frac{z-z^{-1}}{q-q^{-1}}C_{2e2},  \\
&C_e\cdot C_{12e}=\frac{q^{2}z-q^{-2}z^{-1}}{q-q^{-1}}C_e.
\end{align*}
\end{example}

\section{$\imath$Schur duality of type AI}
\label{sec:AI}
\subsection{$\io$quantum group of type AI}
In this section we fix $m\in \mathbb{Z}_{\geq 2}$ and focus on the quantum symmetric pair of type AI with the Satake diagram as below (cf.~\cite[Table 4]{BW18b}):
\begin{equation*}
\label{AI}
\begin{tikzpicture}[scale=1, semithick]
\node (1) [draw,circle,label=below:{$1$},scale=0.5] at (0,0){};
\node (2) [draw,circle,label=below:{$2$},scale=0.5] at (1.3,0){};
\node (3)  at (2.6,0) {$\cdots$} ;
\node (4) [draw,circle,label=below:{$m-2$},scale=0.5] at (3.9,0){};
\node (5) [draw,circle,label=below:{$m-1$},scale=0.5] at (5.2,0){};

\path (1) edge (2)
          (2) edge (3)
          (3) edge (4)
          (4) edge (5);
\end{tikzpicture}
\end{equation*}

For $a\in \mathbb{Z}$ and $b\in \mathbb{N}$, we define
$$[a]=\frac{q^{a}-q^{-a}}{q-q^{-1}},\quad [b]!=\prod_{h=1}^{b}\frac{q^{h}-q^{-h}}{q-q^{-1}},\quad \left[\begin{array}{c}a \\
b
\end{array}\right]=\prod_{h=1}^{b}\frac{q^{a-h+1}-q^{-(a-h+1)}}{q^{h}-q^{-h}}.$$

Let $(a_{ij})$ be the Cartan matrix of type $A_{m-1}$, that is, $a_{ii}=2$, $a_{ij}=-1$ if $|i-j|=1$ and $a_{ij}=0$ if $|i-j|\geq 2$. Let $\U=\U_{q}(\mathfrak{sl}_{m})$ denote the quantum group of type $A_{m-1}$. By definition, $\U$ is the associative algebra over
$\mathbb{Q}(q)$ with generators
$E_{i}, F_{i}, K_{i}^{\pm 1}$ $(1\leq i\leq m-1)$ and the following relations:
\begin{equation*}K_{i}K^{\pm1}_{j}=K^{\pm1}_{j}K_{i}, \quad K_{i}K_{i}^{-1}=1=K_{i}^{-1}K_{i}\
~\quad\hbox{for~all}~ 1\leq i, j\leq m-1,\end{equation*}
\begin{equation*}K_{i}E_{j}K_{i}^{-1}=q^{a_{ij}}E_{j},\quad K_{i}F_{j}K_{i}^{-1}=q^{-a_{ij}}F_{j}
\ ~\quad\hbox{for~all}~ 1\leq i, j\leq m-1,\end{equation*}
\begin{equation*}E_{i}F_{j}-F_{j}E_{i}=\delta_{i,j}\frac{K_{i}-K_{i}^{-1}}{q-q^{-1}},\end{equation*}
\begin{equation*}\sum\limits_{s=0}^{1-a_{ij}}(-1)^{s}\left[\begin{array}{c}1-a_{ij} \\
s
\end{array}\right]E_{i}^{1-a_{ij}-s}E_{j}E_{i}^{s}=0\  ~\quad\hbox{for }i\neq j,\end{equation*}
\begin{equation*}\sum\limits_{s=0}^{1-a_{ij}}(-1)^{s}\left[\begin{array}{c}1-a_{ij} \\
s
\end{array}\right]F_{i}^{1-a_{ij}-s}F_{j}F_{i}^{s}=0\  ~\quad\hbox{for }i\neq j.\end{equation*}

It is well known that $\U$ is a Hopf algebra with the comultiplication $\Delta$ as follows:
\begin{align*}
&\Delta(E_i)=E_i\otimes 1+ K_i \otimes E_i, \\
&\Delta(F_i)=F_i\otimes  K_i^{-1}+1\otimes F_i,\\
&\Delta(K_i^{\pm 1})=K_i^{\pm 1} \otimes K_i^{\pm 1}.
\end{align*}

\begin{definition}\label{1.2}
The $\imath$quantum group $\U^\io(\so_{m})$ of type AI, with a set of  parameters $\{\varsigma_i\mid 1\leq i\leq m-1\}\subset \mathbb{Z}[q, q^{-1}]$, is the $\Q(q)$-subalgebra of $\U$ generated by the following elements:
\begin{equation}
\label{eq:Bi}
B_i=F_i+\va_i E_{i}K_i^{-1}\quad \mathrm{ for }~1\leq i\leq m-1.
\end{equation}
\end{definition}

\begin{remark}
Suppose $\va_i=-1$ for $1\leq i\leq m-1$. When taking the $q\to 1$ limit in $\U^{\io}(\so_{m})$, we see that the generator $B_i$ reduces to $E_{i+1,i}-E_{i,i+1}$, where $E_{j,k}$'s are the $m\times m$ elementary matrices. Therefore, $\U^{\io}(\so_{m})$ reduces to the enveloping algebra $\U(\so_{m})$ of the special orthogonal Lie algebra $\mathfrak{so}_{m}$.
\end{remark}

\subsection{$\imath$Schur duality}
Let $\V=\sum_{i=1}^{m}\Q(q)v_i$ be the natural representation of $\U$ with the action of the generators as follows:
\begin{align*}
E_{i}\cdot v_r&=\delta_{r, i+1}v_{r-1},\\
F_{i}\cdot v_r&=\delta_{r, i}v_{r+1},\\
K_i\cdot v_r&=\begin{cases}
qv_{i}& \hbox {if } r=i, \\
q^{-1} v_{i+1} & \hbox {if } r=i+1,\\
v_{r} &\hbox {else}.
\end{cases}
\end{align*}
Therefore, the action of $B_i$ on $\V$ can be computed by \eqref{eq:Bi}:
\begin{align*}
B_i\cdot v_r=\begin{cases}
v_{i+1}& \hbox {if } r=i, \\
q\va_i v_{i} & \hbox {if } r=i+1,\\
0 &\hbox {else}.
\end{cases}
\end{align*}

\begin{lemma}
$\V^{\otimes n}$ is a left $\Ui(\so_{m})$-module via $\Delta$.
\end{lemma}

For $i=2,\ldots,m$, we set $$\tau_i:=\prod_{j=1}^{i-1}(-\va_j)$$ and $\tau_1=1.$
Then we have the following lemma.
\begin{proposition}
\label{Bact-add}
$\V^{\otimes n}$ is a right $\B_n(q,q^{m})$-module with the action given by
\begin{align*}
&v_{a_1}\otimes\cdots \otimes v_{a_n}\cdot H_j\\
=&\left\{\begin{aligned}
&qv_{a_1}\otimes\cdots \otimes v_{a_n}\ \ & \text{ if } a_j=a_{j+1}, \\
&\cdots\otimes v_{a_{j+1}}\otimes v_{a_j}\otimes \cdots\ \ &\text{ if } a_{j}>a_{j+1}, \\
&\cdots \otimes v_{a_{j+1}}\otimes v_{a_j}\otimes \cdots+(q-q^{-1})v_{a_1}\otimes\cdots \otimes v_{a_n}\ \ & \text{ if } a_{j}<a_{j+1},
\end{aligned}
\right. \\
&v_{a_1}\otimes v_{a_2}\otimes\cdots\otimes v_{a_n}\cdot \Qy=\delta_{a_{1},a_2} \tau_{a_1}\left(\sum_{i=1}^{m}\tau_i^{-1} q^{m-2i+1}v_{i}\otimes v_{i}\right)\otimes v_{a_3}\otimes \cdots \otimes v_{a_n}.
\end{align*}
\end{proposition}
\begin{proof}
By \cite{Jim86}, the action of $H_i$ satisfies relations (Q1)-(Q3) in Definition \ref{def:qB1}. In order to verify the relation (Q4), noting that the action of $\Qy$ depends solely on the first two tensor factors, it suffices to show that $$v_{a_1}\otimes v_{a_2}\cdot\Qy^2=\frac{q^{m}-q^{-m}}{q-q^{-1}}v_{a_1}\otimes v_{a_2}\cdot\Qy.$$ We have
\begin{align*}
&v_{a_1}\otimes v_{a_2}\cdot\Qy^2\\
=&\delta_{a_{1},a_2} \tau_{a_1}\left(\sum_{i=1}^{m}\tau_i^{-1} q^{m-2i+1}v_{i}\otimes v_{i}\right)\cdot\Qy\\
=&\delta_{a_{1},a_2} \tau_{a_1}\sum_{i=1}^{m}\tau^{-1}_{i} q^{m-2i+1}\left(\tau_i\sum_{j=1}^{m}\tau_{j}^{-1} q^{m-2j+1}v_{j}\otimes v_{j}\right)\\
=&\delta_{a_{1},a_2} \tau_{a_1}\sum_{j=1}^{m}\tau_{j}^{-1}q^{m-2j+1}\left(\sum_{i=1}^{m}q^{m-2i+1}\right)v_{j}\otimes v_{j}\\
=&\frac{q^{m}-q^{-m}}{q-q^{-1}}v_{a_1}\otimes v_{a_2}\cdot\Qy.
\end{align*}
The relation (Q5) can be easily verified. In order to verify the relation (Q6), it suffices to show that
$$v_{a_1}\otimes v_{a_2}\otimes v_{r}\cdot\Qy H_{2}\Qy=q^{m}v_{a_1}\otimes v_{a_2}\otimes v_{r}\cdot\Qy.$$ We have
\begin{align*}
&v_{a_1}\otimes v_{a_2}\otimes v_{r}\cdot\Qy H_{2}\Qy\\
=&\delta_{a_{1},a_2} \tau_{a_1}\left(\sum_{i=1}^{m}\tau_i^{-1} q^{m-2i+1}v_{i}\otimes v_{i}\otimes v_{r}\right)\cdot H_{2}\Qy\\
=&\delta_{a_{1},a_2}\tau_{a_1} \sum_{i=1}^{r-1}\tau_{i}^{-1} q^{m-2i+1}\left(v_{i}\otimes v_{r}\otimes v_{i}+(q-q^{-1})v_{i}\otimes v_{i}\otimes v_{r}\right)\cdot\Qy\\
&+\delta_{a_{1},a_2}\tau_{a_1}\tau_{r}^{-1} q^{m-2r+1}\cdot qv_{r}\otimes v_{r}\otimes v_{r} \cdot\Qy\\
&+\delta_{a_{1},a_2}\tau_{a_1}\sum_{i=r+1}^{m}\tau_i^{-1} q^{m-2i+1}v_{i}\otimes v_{r}\otimes v_{i}\cdot\Qy\\
=&\delta_{a_{1},a_2}\tau_{a_1} \sum_{i=1}^{r-1}\tau_{i}^{-1} q^{m-2i+1}(q-q^{-1})\left(\sum_{j=1}^{m}\tau_i\tau_j^{-1} q^{m-2j+1}v_{j}\otimes v_{j}\otimes v_{r}\right)\\
&+\delta_{a_{1},a_2}\tau_{a_1}\tau_r^{-1} q^{m-2r+2}\left(\sum_{j=1}^{m}\tau_r\tau_j^{-1} q^{m-2j+1}v_{j}\otimes v_{j}\otimes v_{r}\right)\\
=&\delta_{a_{1},a_2}\sum_{j=1}^{m}\tau_{a_1}\tau_j^{-1} q^{m-2j+1}\left(\sum_{i=1}^{r-1}q^{m-2i+1}(q-q^{-1})+q^{m-2r+2}\right)v_{j}\otimes v_{j}\otimes v_{r}\\
=&q^{m}v_{a_1}\otimes v_{a_2}\otimes v_{r}\cdot\Qy.
\end{align*}

The relation (Q7) can be easily verified. From the action of $H_j$ we can easily obtain the action of $H_j^{-1}$ as follows:
\begin{align*}
&v_{a_1}\otimes\cdots \otimes v_{a_n}\cdot H_j^{-1}\\
=&\left\{\begin{aligned}
&q^{-1}v_{a_1}\otimes\cdots \otimes v_{a_n},\ \ & \text{ if } a_j=a_{j+1}, \\
&\cdots \otimes v_{a_{j+1}}\otimes v_{a_j}\otimes \cdots+(q^{-1}-q)v_{a_1}\otimes\cdots \otimes v_{a_n},\ \ &\text{ if } a_{j}>a_{j+1},\\
&\cdots\otimes v_{a_{j+1}}\otimes v_{a_j}\otimes \cdots,\ \ & \text{ if } a_{j}<a_{j+1}.
\end{aligned}
\right.
\end{align*}

In order to verify $\Qy(H_{2}H_{3}H_{1}^{-1}H_{2}^{-1})\Qy=\Qy(H_{2}H_{3}H_{1}^{-1}H_{2}^{-1})\Qy H_{2}H_{3}H_{1}^{-1}H_{2}^{-1}$, it suffices to show that
\begin{equation}\label{sophicated equality}
v_{a_1}\otimes v_{a_1}\otimes v_{k}\otimes v_{l}\cdot\Qy(H_{2}H_{3}H_{1}^{-1}H_{2}^{-1})\Qy H_{2}H_{1}=v_{a_1}\otimes v_{a_1}\otimes v_{k}\otimes v_{l}\cdot\Qy(H_{2}H_{3}H_{1}^{-1}H_{2}^{-1})\Qy H_{2}H_{3}.
\end{equation}

When $k< l$, we have
\begin{align*}
&v_{a_1}\otimes v_{a_1}\otimes v_{k}\otimes v_{l}\cdot\Qy(H_{2}H_{3}H_{1}^{-1}H_{2}^{-1})\Qy\\
=&\left(\sum_{i=1}^{m}\tau_{a_1}\tau_i^{-1} q^{m-2i+1}v_{i}\otimes v_{i}\otimes v_{k}\otimes v_{l}\right)\cdot H_{2}H_{3}H_{1}^{-1}H_{2}^{-1}\Qy\\
=&\sum_{i=1}^{k-1}\tau_{a_1}\tau_i^{-1} q^{m-2i+1}\left(v_{i}\otimes v_{k}\otimes v_{i}\otimes v_{l}+(q-q^{-1})v_{i}\otimes v_{i}\otimes v_{k}\otimes v_{l}\right)\cdot H_{3}H_{1}^{-1}H_{2}^{-1}\Qy\\
&+\tau_{a_1}\tau_k^{-1} q^{m-2k+1}\cdot qv_{k}\otimes v_{k}\otimes v_{k}\otimes v_{l}\cdot H_{3}H_{1}^{-1}H_{2}^{-1}\Qy\\
&+\sum_{i=k+1}^{m}\tau_{a_1}\tau_i^{-1} q^{m-2i+1}v_{i}\otimes v_{k}\otimes v_{i}\otimes v_{l}\cdot H_{3}H_{1}^{-1}H_{2}^{-1}\Qy\\
=&\tau_{a_1}\tau_k^{-1} q^{m-2k+1}\cdot q(q-q^{-1})q^{-2}v_{k}\otimes v_{k}\otimes v_{k}\otimes v_{l}\cdot \Qy\\
&+\sum_{i=k+1}^{l-1}\tau_{a_1}\tau_i^{-1} q^{m-2i+1}(q-q^{-1})(q^{-1}-q)v_{i}\otimes v_{i}\otimes v_{k}\otimes v_{l}\cdot \Qy\\
&+\tau_{a_1}\tau_l^{-1} q^{m-2l+1}\cdot q(q^{-1}-q)v_{l}\otimes v_{l}\otimes v_{k}\otimes v_{l}\cdot \Qy\\
=&\sum_{j=1}^{m}\tau_{a_1}\tau_j^{-1} q^{m-2j+1}v_{j}\otimes v_{j}\otimes v_{k}\otimes v_{l}\\
&\times\left(q^{m-2k}(q-q^{-1})-(q-q^{-1})^{2}\sum_{i=k+1}^{l-1}q^{m-2i+1}+q^{m-2l+2}(q^{-1}-q)\right)\\
=&0.
\end{align*}
When $k> l$, it can be similarly shown that $$v_{a_1}\otimes v_{a_1}\otimes v_{k}\otimes v_{l}\cdot\Qy(H_{2}H_{3}H_{1}^{-1}H_{2}^{-1})\Qy=0.$$ When $k=l$, we have
\begin{align*}
&v_{a_1}\otimes v_{a_1}\otimes v_{k}\otimes v_{k}\cdot\Qy(H_{2}H_{3}H_{1}^{-1}H_{2}^{-1})\Qy\\
=&\left(\sum_{i=1}^{m}\tau_{a_1}\tau_i^{-1} q^{m-2i+1}v_{i}\otimes v_{i}\otimes v_{k}\otimes v_{k}\right)\cdot H_{2}H_{3}H_{1}^{-1}H_{2}^{-1}\Qy\\
=&\sum_{i=1}^{k-1}\tau_{a_1}\tau_i^{-1}q^{m-2i+1}v_{k}\otimes v_{k}\otimes v_{i}\otimes v_{i}\cdot \Qy+\tau_{a_1}\tau_k^{-1}q^{m-2k+1}v_{k}\otimes v_{k}\otimes v_{k}\otimes v_{k}\cdot \Qy\\
&+\sum_{i=k+1}^{m}\tau_{a_1}\tau_i^{-1}q^{m-2i+1}v_{k}\otimes v_{k}\otimes v_{i}\otimes v_{i}\cdot \Qy\\
=&\sum_{i=1}^{m}\tau_{a_1}\tau_i^{-1}sq^{m-2i+1}v_{k}\otimes v_{k}\otimes v_{i}\otimes v_{i}\cdot \Qy\\
=&\sum_{i,j=1}^{m}\tau_{a_1}\tau_k\tau_i^{-1}\tau_j^{-1}q^{2m-2i-2j+2}v_{j}\otimes v_{j}\otimes v_{i}\otimes v_{i}.
\end{align*}
It is straightforward to show that
\begin{align*}
&\sum_{i,j=1}^{m}\tau_{a_1}\tau_k\tau_i^{-1}\tau_j^{-1}q^{2m-2i-2j+2}v_{j}\otimes v_{j}\otimes v_{i}\otimes v_{i}\cdot H_{2}H_{1}\\
=&\sum_{i,j=1}^{m}\tau_{a_1}\tau_k\tau_i^{-1}\tau_j^{-1}q^{2m-2i-2j+2}v_{j}\otimes v_{j}\otimes v_{i}\otimes v_{i}\cdot H_{2}H_{3}.
\end{align*}
Therefore, \eqref{sophicated equality} holds. The equality $$H_{2}H_{3}H_{1}^{-1}H_{2}^{-1}\Qy(H_{2}H_{3}H_{1}^{-1}H_{2}^{-1})\Qy=\Qy(H_{2}H_{3}H_{1}^{-1}H_{2}^{-1})\Qy$$ can be proved similarly. We are done.
\end{proof}

\begin{remark}
\label{remark:ortho}
Let $O_{m}$ and $SO_{m}$ denote the orthogonal group and special orthogonal group, respectively. As shown in \cite{Br37}, the Brauer algebra surjects onto $End_{O_{m}}(V^{\otimes k})$ for all $k\in \Z_{>0}$, where $V$ is the natural representation of $O_m$. But if one replaces $O_{m}$ by $SO_{m}$, then we have the following result (see \cite[\S 5.1.3]{LZ06}):
\begin{align*}
&\text{If $m$ is odd, then $End_{O_{m}}(V^{\otimes k})=End_{SO_{m}}(V^{\otimes k})$ for all $k$.}\\
&\text{If $m$ is even, then $End_{O_{m}}(V^{\otimes k})=End_{SO_{m}}(V^{\otimes k})$ if and only if $m-1\geqslant 2k$.}
\end{align*}
\end{remark}


\begin{theorem}
\label{prop:dcp}
$(1)$ The left action of $\Ui(\so_{m})$ on $\V^{\otimes n}$ commutes with the right action of $\B_n(q,q^{m})$ defined in Proposition~\ref{Bact-add}:
\[
\Ui(\so_{m}) \stackrel{\Psi}{\curvearrowright} \V^{\otimes n} \stackrel{\Phi}{\curvearrowleft} \B_n(q,q^{m}).
\]
$(2)$ When $m$ is odd or $m$ is even with $m-1\geqslant 2n$, the following double centralizer property holds:
\begin{align*}
\Psi(\Ui(\so_{m}))=\End_{\B_n(q,q^{m})} (\V^{\otimes n}),\\
\Phi(\B_n(q,q^{m}))=\End_{\Ui(\so_{m})} (\V^{\otimes n}).
\end{align*}
\end{theorem}

\begin{proof}
$(1)$ By the Jimbo duality in \cite{Jim86}, we know that the action of $\U$ commutes with the action of $H_i$ for $1\le i \le n-1$. Thus, to show the commuting actions of $\Ui(\so_{m})$ and $\B_n(q,q^{m})$, it remains to check the commutativity of the actions of $B_i$ ($1\leq i\leq m-1$) and $\Qy$.

Thanks to $\Delta(B_i)=B_i\otimes K_i^{-1}+ 1 \otimes B_i$ and the fact that the action of $\Qy$ depends solely on the first two tensor factors, it suffices to consider $n=2$. By a direct calculation, it can be shown that
\begin{align*}
B_i\cdot (v_{a_1}\otimes v_{a_2}\cdot \Qy)=0=(B_i\cdot v_{a_1}\otimes v_{a_2})\cdot \Qy.
\end{align*}
We omit the details.

$(2)$ The double centralizer property is equivalent to the multiplicity-free decomposition of $\V^{\otimes n}$ as an $\Ui(\so_{m})$-$\B_n(q,q^{m})$-bimodule. According to \cite[\S 4.5]{We12a} or \cite[\S 4]{N18}, the $q$-Brauer algebra $\B_n(q,q^{m})$ is semisimple when $q$ is generic; moreover, when taking the $q\to1$ limit, the cell module of $\B_n(q,q^{m})$ recovers the cell module of the classical Brauer algebra defined in \cite{GL96}. Thus, the proof of the double centralizer property reduces by a deformation argument to the $q=1$ setting. When taking the $q\to1$ limit and $\va_i=-1$ $(1\leq i\leq m-1)$, $\U^{\io}(\so_{m})$ becomes the enveloping algebra of the special orthogonal Lie algebra $\mathfrak{so}_{m}$, $\V$ becomes its natural representation. By lifting, $\V$ can also be regarded as a representation of the special orthogonal group $SO_{m}$. Moreover, according to Remark \ref{remark:ortho}, when $m$ is odd or $m$ is even with $m-1\geqslant 2n$, we have $$End_{\so_{m}}(\V^{\otimes n})=End_{SO_{m}}(\V^{\otimes n})=End_{O_{m}}(\V^{\otimes n}).$$ The multiplicity-free decomposition of $\V^{\otimes n}$ in this case has been established in \cite{Br37}, \cite{Br56a} and \cite{Br56b}. We are done.
\end{proof}

\begin{remark}
The condition on $m$ required in Theorem \ref{prop:dcp}$(2)$ can be removed if we enlarge the $\imath$quantum group to an algebra generated by $\Ui(\so_{m})$ and $\varrho$ over $\Q(q)$ with the following relations:
\begin{equation*}
\label{def:Ui(O)}
\varrho^2=1,\quad \varrho B_i=(-1)^{\delta_{1,i}}B_i\varrho \quad\mbox{for $1\leq i\leq m-1$}.
\end{equation*}
We put the action of $\varrho$ on $\V$ by
$$
\varrho\cdot v_r=\begin{cases}
-v_{1}& \hbox {if } r=1, \\
v_{r} & \hbox {if } r>1.
\end{cases}
$$
One can show this action commutes with the $q$-Brauer algebra action and when taking the $q\to1$ limit, the new algebra reduces to $\U(\so_{m})\oplus \U(\so_{m})\varrho$.
\end{remark}


\section{$\imath$Schur duality of type AII}
\label{sec:AII}
\subsection{$\io$quantum group of type AII}
In this section we fix $m\in \mathbb{Z}_{\geq 1}$ and focus on the quantum symmetric pair of type AII with the Satake diagram as below  (cf.~\cite[Table 4]{BW18b}):
\begin{equation*}
\label{AII}
\begin{tikzpicture}[scale=1, semithick]
\node (1) [fill,circle,label=below:{$1$},scale=0.5] at (0,0){};
\node (2) [circle,draw,label=below:{$2$},scale=0.5] at (1.3,0){};
\node (3)  [fill,circle,label=below:{$3$},scale=0.5]at (2.6,0) {};
\node (4)  at (3.9,0){$\cdots$};
\node (5) [circle,draw,label=below:{$2m-2$},scale=0.5] at (5.2,0){};
\node (6) [fill,circle,label=below:{$2m-1$},scale=0.5] at (6.5,0){};

\path (1) edge (2)
          (2) edge (3)
          (3) edge (4)
          (4) edge (5)
          (5) edge (6);
\end{tikzpicture}
\end{equation*}

Let $(a_{ij})$ be the Cartan matrix of type $A_{2m-1}$ and let $\U_{q}(\mathfrak{sl}_{2m})$ denote the corresponding quantum group over $\Q(q)$ (cf. $\S4.1$). Recall that in \cite[\S37.1.3]{Lu94}, Lusztig has defined an algebra automorphism $T''_{i, 1}: \U_{q}(\mathfrak{sl}_{2m}) \rightarrow \U_{q}(\mathfrak{sl}_{2m})$, which we shall write $T_i$ for simplicity. We need the following action of it on $\U_{q}(\mathfrak{sl}_{2m})$:
$$T_{i} (E_i)=-F_i K_{i}, \quad T_{i}(E_j)=  \sum_{r+s=-a_{ij}} (-1)^r q^{-r} E^{(s)}_i E_j E^{(r)}_i \quad\mathrm{ for }~j\neq i,$$
where $E^{(s)}_i=E_i^{s}/[s]!$.
\begin{definition}\label{def:AII}
The $\io$quantum group $\Ui(\spin_{2m})$ of type AII, with a set of parameters $\{\varsigma_i\mid i=2,4,\ldots,2m-2\}\subset \mathbb{Z}[q, q^{-1}]$, is the $\Q(q)$-subalgebra of $\U_{q}(\mathfrak{sl}_{2m})$ generated by the following elements:
$$B_i=F_i+\va_i T_{i-1}T_{i+1}(E_{i})K_i^{-1}\quad \mathrm{ for }~i=2,4,\ldots,2m-2,$$
$$E_{j}, F_{j}, K_{j}^{\pm 1}\quad \mathrm{ for }~j=1,3,\ldots,2m-1.$$
\end{definition}

\begin{remark}
Let $E_{i,j}$ denote the $2m\times 2m$ elementary matrices and $M$ be a $2m\times 2m$ skew-symmetric quasi-diagonal matrix $M=diag\{J,\ldots,J\}$ with $J=\begin{pmatrix} 0 & 1 \\ -1 & 0 \end{pmatrix}$.

Suppose $\va_i=-1$ ($i$ even). When taking the $q\to 1$ limit in $\Ui(\spin_{2m})$, we see that the generators $E_j$ and $F_j$ ($j \text{ odd}$) reduce to matrices $E_{j,j+1}$ and $E_{j+1,j}$ respectively. Moreover, $B_i$ ($i \text{ even}$) reduces to $E_{i+1,i}+E_{i-1,i+2}$. 

Therefore, $\Ui(\spin_{2m})$ indeed reduces to the enveloping algebra of the symplectic Lie algebra $\mathfrak{sp}_{2m}$, which is characterized as a Lie algebra consisting of all $2m\times 2m$ matrices $X$ satisfying the condition $X^tM+MX=0$.
\end{remark}

In fact, for $i=2,4,\ldots,2m-2$, we have
\begin{align*}
T_{i-1}&T_{i+1}(E_{i})\\
&=E_{i+1}E_{i-1}E_{i}-q^{-1}E_{i-1}E_{i}E_{i+1}-q^{-1}E_{i+1}E_{i}E_{i-1}+q^{-2}E_{i}E_{i-1}E_{i+1}.
\end{align*}
Let $\W=\sum_{i=1}^{2m}\Q(q)v_i$ be the natural representation of $\U_{q}(\mathfrak{sl}_{2m})$. By a direct calculation we see that the action of $B_{2l}$ ($l=1,2,\ldots,m-1$) on $\W$ is given by
\begin{align*}
B_{2l}\cdot v_r=\begin{cases}
v_{2l+1}& \hbox {if } r=2l, \\
-q^{-1}\va_{2l} v_{2l-1} & \hbox {if } r=2l+2,\\
0 &\hbox {else}.
\end{cases}
\end{align*}




\subsection{$\imath$Schur duality}

For $i=2,\ldots,m$, we set $$\kappa_i:=\prod_{j=1}^{i-1}(-\va_{2j})$$ and $\kappa_1=1$. Then the following lemma gives a right $\B_n(-q^{-1},q^{2m})$-module structure on $\W^{\otimes n}$.
\begin{proposition}
\label{Bact2}
There is a right action of $\B_n(-q^{-1},q^{2m})$ on $\W^{\otimes n}$ via
\begin{align*}
&v_{a_1}\otimes\cdots \otimes v_{a_n}\cdot H_k\\
=&\left\{\begin{aligned}
&qv_{a_1}\otimes\cdots \otimes v_{a_n}\ \ & \text{ if } a_k=a_{k+1}, \\
&\cdots\otimes v_{a_{k+1}}\otimes v_{a_k}\otimes \cdots\ \ &\text{ if } a_{k}>a_{k+1}, \\
&\cdots \otimes v_{a_{k+1}}\otimes v_{a_k}\otimes \cdots+(q-q^{-1})v_{a_1}\otimes\cdots \otimes v_{a_n}\ \ & \text{ if } a_{k}<a_{k+1},
\end{aligned}
\right.
\end{align*}
and
\begin{align*}
&v_{a_1}\otimes v_{a_2}\otimes\cdots\otimes v_{a_n}\cdot \Qy\\
=&\begin{cases}
\sum\limits_{j=1}^{m} \kappa_i\kappa_j^{-1}q^{2m+1-3i-j}(v_{2j-1}\otimes v_{2j}-qv_{2j}\otimes v_{2j-1})\otimes v_{a_3}\otimes \cdots \otimes v_{a_n} \\
\hspace{8cm}\hbox { if } a_1=2i-1, a_2=2i, \\
(-q)v_{a_2}\otimes v_{a_1}\otimes v_{a_3}\otimes \cdots \otimes v_{a_n}\cdot \Qy \hspace{2.35cm}\hbox { if } a_1=2i, a_2=2i-1,\\
0 \hspace{7.8cm}\hbox { else},
\end{cases}
\end{align*}
where $i=1,2,\ldots,m$.
\end{proposition}

\begin{proof}
Noting that the action of $\Qy$ depends solely on the first two tensor factors, in order to verify the relation (Q4) in Definition \ref{def:qB1} it suffices to show that $$v_{1}\otimes v_{2}\cdot\Qy^2=\frac{q^{2m}-q^{-2m}}{q-q^{-1}}v_{1}\otimes v_{2}\cdot\Qy.$$

We have
\begin{align*}
v_1\otimes v_2\cdot e^2=&\sum_{j=1}^{m} \kappa_1\kappa_j^{-1}q^{2m-2-j}(v_{2j-1}\otimes v_{2j}-q v_{2j}\otimes v_{2j-1})\cdot \Qy\\
=&\sum_{j=1}^{m} \kappa_1\kappa_j^{-1}q^{2m-2-j}(1+q^2)(v_{2j-1}\otimes v_{2j})\cdot \Qy \\
=&\sum_{j=1}^{m} \kappa_1\kappa_j^{-1}q^{2m-2-j}(1+q^2)\kappa_j\kappa_1^{-1}q^{-3(j-1)}(v_1\otimes v_2)\cdot \Qy\\
=&(1+q^2)\left(\sum_{j=1}^{m}q^{2m+1-4j}\right)(v_1\otimes v_2)\cdot \Qy \\
=&\frac{q^{2m}-q^{-2m}}{q-q^{-1}} (v_1\otimes v_2)\cdot \Qy.
\end{align*}
The relation (Q5) can be easily verified. In order to verify the relation (Q6), it suffices to show that
$$v_{1}\otimes v_{2}\otimes v_{r}\cdot\Qy H_{2}\Qy=q^{2m}v_{1}\otimes v_{2}\otimes v_{r}\cdot\Qy.$$
When $r=2k-1$, we have
\begin{align*}
&v_{1}\otimes v_{2}\otimes v_{2k-1}\cdot\Qy H_{2}\Qy\\
=&\sum_{j=1}^{m} \kappa_1\kappa_j^{-1}q^{2m-2-j}(v_{2j-1}\otimes v_{2j}\otimes v_{2k-1}-q v_{2j}\otimes v_{2j-1}\otimes v_{2k-1})\cdot H_{2}\Qy\\
=&\sum_{j=1}^{k-1} q^{2m+1-4j}(q-q^{-1})v_{1}\otimes v_{2}\otimes v_{2k-1}\cdot\Qy\\
&-q\sum_{j=1}^{k-1}q^{2m+1-4j}(q-q^{-1})(-q)v_{1}\otimes v_{2}\otimes v_{2k-1}\cdot\Qy\\
&-q^{2m+2-4k}(-q^{2})v_{1}\otimes v_{2}\otimes v_{2k-1}\cdot\Qy\\
=&\left((q-q^{-1})(1+q^{2})\sum_{j=1}^{k-1}q^{2m+1-4j}+q^{2m+4-4k}\right)v_{1}\otimes v_{2}\otimes v_{2k-1}\cdot\Qy\\
=&q^{2m}v_{1}\otimes v_{2}\otimes v_{r}\cdot\Qy.
\end{align*}
When $r=2k$, it can be proved similarly. (Q7) can be easily verified.

In order to verify $\Qy(H_{2}H_{3}H_{1}^{-1}H_{2}^{-1})\Qy=\Qy(H_{2}H_{3}H_{1}^{-1}H_{2}^{-1})\Qy H_{2}H_{3}H_{1}^{-1}H_{2}^{-1}$, it suffices to show that
\begin{equation}\label{sophicated equality-adddd}
v_{1}\otimes v_{2}\otimes v_{p}\otimes v_{q}\cdot\Qy(H_{2}H_{3}H_{1}^{-1}H_{2}^{-1})\Qy H_{2}H_{1}=v_{1}\otimes v_{2}\otimes v_{p}\otimes v_{q}\cdot\Qy(H_{2}H_{3}H_{1}^{-1}H_{2}^{-1})\Qy H_{2}H_{3}.
\end{equation}
When $p=2k$ and $q=2l$ with $k<l$, we have
\begin{align*}
&v_{1}\otimes v_{2}\otimes v_{2k}\otimes v_{2l}\cdot\Qy(H_{2}H_{3}H_{1}^{-1}H_{2}^{-1})\Qy\\
=&\sum_{j=1}^{m} \kappa_1\kappa_j^{-1}q^{2m-2-j}(v_{2j-1}\otimes v_{2j}\otimes v_{2k}\otimes v_{2l}-
\\&\hspace{5cm}q v_{2j}\otimes v_{2j-1}\otimes v_{2k}\otimes v_{2l})\cdot H_{2}H_{3}H_{1}^{-1}H_{2}^{-1}\Qy\\
=&-(q-q^{-1})^{2}\sum_{j=k+1}^{l-1}\kappa_1\kappa_j^{-1}q^{2m-2-j}v_{2j-1}\otimes v_{2j}\otimes v_{2k}\otimes v_{2l}\cdot \Qy\\
&+\kappa_1\kappa_l^{-1}q^{2m-2-l}\cdot q(q^{-1}-q)v_{2l-1}\otimes v_{2l}\otimes v_{2k}\otimes v_{2l}\cdot \Qy\\
&-\kappa_1\kappa_k^{-1}q^{2m-1-k}(q-q^{-1})q^{-1}v_{2k}\otimes v_{2k-1}\otimes v_{2k}\otimes v_{2l}\cdot \Qy\\
&-\kappa_1\kappa_k^{-1}q^{2m-1-k}(q-q^{-1})^{2}q^{-1}v_{2k-1}\otimes v_{2k}\otimes v_{2k}\otimes v_{2l}\cdot \Qy\\
&+(q-q^{-1})^{2}\sum_{j=k+1}^{l}\kappa_1\kappa_j^{-1}q^{2m-1-j}v_{2j}\otimes v_{2j-1}\otimes v_{2k}\otimes v_{2l}\cdot \Qy\\
=&A v_{1}\otimes v_{2}\otimes v_{2k}\otimes v_{2l}\cdot \Qy,
\end{align*}
where \begin{align*}A=&-(q-q^{-1})^{2}\sum_{j=k+1}^{l-1}q^{2m+1-4j}+(q^{-1}-q)q^{2m+2-4l}\\
&+(q-q^{-1})q^{2m+2-4k}-(q-q^{-1})^{2}q^{2m+1-4k}-(q-q^{-1})^{2}\sum_{j=k+1}^{l}q^{2m+3-4j}\\
=&-(q-q^{-1})^{2}(1+q^{2})\sum_{j=k+1}^{l-1}q^{2m+1-4j}
\\&\hspace{5cm}+(q-q^{-1})q^{2m-4k}+(q^{-1}-q)q^{2m+4-4l}\\
=&0.
\end{align*}
Therefore, in this case we have $v_{1}\otimes v_{2}\otimes v_{2k}\otimes v_{2l}\cdot\Qy(H_{2}H_{3}H_{1}^{-1}H_{2}^{-1})\Qy=0.$

In a similar way, we can show that $v_{1}\otimes v_{2}\otimes v_{2k}\otimes v_{2l}\cdot\Qy(H_{2}H_{3}H_{1}^{-1}H_{2}^{-1})\Qy=0$ when $k\geq l$, $v_{1}\otimes v_{2}\otimes v_{2k-1}\otimes v_{2l-1}\cdot\Qy(H_{2}H_{3}H_{1}^{-1}H_{2}^{-1})\Qy=0$ for any $k, l$, and $v_{1}\otimes v_{2}\otimes v_{2k}\otimes v_{2l-1}\cdot\Qy(H_{2}H_{3}H_{1}^{-1}H_{2}^{-1})\Qy=0=v_{1}\otimes v_{2}\otimes v_{2k-1}\otimes v_{2l}\cdot\Qy(H_{2}H_{3}H_{1}^{-1}H_{2}^{-1})\Qy$ when $k\neq l$.

When $p=2k$ and $q=2k-1$, we have
\begin{align*}
&v_{1}\otimes v_{2}\otimes v_{2k}\otimes v_{2k-1}\cdot\Qy(H_{2}H_{3}H_{1}^{-1}H_{2}^{-1})\Qy\\
=&\sum_{j=1}^{k-1}\kappa_1\kappa_j^{-1}q^{2m-2-j}v_{2k}\otimes v_{2k-1}\otimes v_{2j-1} \otimes v_{2j}\cdot \Qy\\
&+\kappa_1\kappa_k^{-1}q^{2m-2-k}v_{2k}\otimes v_{2k-1}\otimes v_{2k-1} \otimes v_{2k}\cdot \Qy\\
&+\sum_{j=k+1}^{m}\kappa_1\kappa_j^{-1}q^{2m-2-j}v_{2k}\otimes v_{2k-1}\otimes v_{2j-1} \otimes v_{2j}\cdot \Qy\\
&-\sum_{j=1}^{k-1}\kappa_1\kappa_j^{-1}q^{2m-1-j}v_{2k}\otimes v_{2k-1}\otimes v_{2j} \otimes v_{2j-1}\cdot \Qy\\
&-\kappa_1\kappa_k^{-1}sq^{2m-1-k}v_{2k}\otimes v_{2k-1}\otimes v_{2k} \otimes v_{2k-1}\cdot \Qy\\
&-\sum_{j=k+1}^{m}\kappa_1\kappa_j^{-1}q^{2m-1-j}v_{2k}\otimes v_{2k-1}\otimes v_{2j} \otimes v_{2j-1}\cdot \Qy\\
=&\sum_{j=1}^{m}\kappa_1\kappa_j^{-1}q^{2m-2-j}v_{2k}\otimes v_{2k-1}\otimes (v_{2j-1} \otimes v_{2j}-q v_{2j} \otimes v_{2j-1})\cdot \Qy\\
=&B\cdot \sum_{i, j=1}^{m}\kappa^2_1\kappa_i^{-1}\kappa_j^{-1}q^{4m-4-i-j}(v_{2i-1} \otimes v_{2i}-q v_{2i} \otimes v_{2i-1}) \otimes
\\&\hspace{5cm}(v_{2j-1} \otimes v_{2j}-q v_{2j} \otimes v_{2j-1}),
\end{align*}
where $B=-q^{4-3k}\kappa_k\kappa_1^{-1}$. By a direct calculation, we can show that
\begin{align*}
&\sum_{i, j=1}^{m}\kappa^2_1\kappa_i^{-1}\kappa_j^{-1}q^{4m-4-i-j}(v_{2i-1} \otimes v_{2i}-q v_{2i} \otimes v_{2i-1}) \otimes
\\&\hspace{5cm}(v_{2j-1} \otimes v_{2j}-q v_{2j} \otimes v_{2j-1})\cdot H_{2}H_{1}\\
=&\sum_{i, j=1}^{m}\kappa^2_1\kappa_i^{-1}\kappa_j^{-1}q^{4m-4-i-j}(v_{2i-1} \otimes v_{2i}-q v_{2i} \otimes v_{2i-1}) \otimes
\\&\hspace{5cm}(v_{2j-1} \otimes v_{2j}-q v_{2j} \otimes v_{2j-1})\cdot H_{2}H_{3}.
\end{align*}
Therefore, \eqref{sophicated equality-adddd} holds when $p=2k$ and $q=2k-1$. Similarly, we can show that \eqref{sophicated equality-adddd} holds when $p=2k-1$ and $q=2k$. The equality $$H_{2}H_{3}H_{1}^{-1}H_{2}^{-1}\Qy(H_{2}H_{3}H_{1}^{-1}H_{2}^{-1})\Qy=\Qy(H_{2}H_{3}H_{1}^{-1}H_{2}^{-1})\Qy$$ can be proved similarly. We are done.
\end{proof}

\begin{theorem}
The left action of $\Ui(\spin_{2m})$ on $\W^{\otimes n}$ commutes with the right action defined in Proposition~\ref{Bact2}:
\[
\Ui(\spin_{2m}) \stackrel{\Psi'}{\curvearrowright} \W^{\otimes n} \stackrel{\Phi'}{\curvearrowleft} \B_n(-q^{-1},q^{2m}).
\]
Moreover, the following double centralizer property holds:
\begin{align*}
&\Psi'(\Ui(\spin_{2m}))=\End_{\B_n(-q^{-1},q^{2m})} (\W^{\otimes n}),\\
&\Phi'(\B_n(-q^{-1},q^{2m}))=\End_{\Ui(\spin_{2m})} (\W^{\otimes n}).
\end{align*}
\end{theorem}

\begin{proof}
By the Jimbo duality in \cite{Jim86}, we know that the action of $\U_{q}(\mathfrak{sl}_{2m})$ commutes with the action of $H_k$ for $1\le k \le n-1$. Thus, to show the commuting actions of $\Ui(\spin_{2m})$ and $\B_n(-q^{-1},q^{2m})$, it remains to check the commutativity of the actions of the generators of $\Ui(\spin_{2m})$ and $\Qy$. Noting that the action of $\Qy$ depends solely on the first two tensor factors, it suffices to consider $n=2$.

We have $E_1\cdot v_{1}\otimes v_{2}=q v_{1}\otimes v_{1}$, $E_1\cdot v_{2}\otimes v_{1}=v_{1}\otimes v_{1}$ and $E_1\cdot v_{k}\otimes v_{l}=0$ for $\{k,l\}\neq \{1, 2\}$, which imply that
$$E_1\cdot (v_{a_1}\otimes v_{a_2}\cdot \Qy)=0=(E_1\cdot v_{a_1}\otimes v_{a_2})\cdot \Qy.$$
Similarly, we can show that $$E_j\cdot (v_{a_1}\otimes v_{a_2}\cdot \Qy)=0=(E_j\cdot v_{a_1}\otimes v_{a_2})\cdot \Qy,$$
$$F_j\cdot (v_{a_1}\otimes v_{a_2}\cdot \Qy)=0=(F_j\cdot v_{a_1}\otimes v_{a_2})\cdot \Qy,$$
$$K_{j}^{\pm 1}\cdot (v_{a_1}\otimes v_{a_2}\cdot \Qy)=v_{a_1}\otimes v_{a_2}\cdot \Qy=(K_{j}^{\pm 1}\cdot v_{a_1}\otimes v_{a_2})\cdot \Qy$$
for $j=1,3,\ldots,2m-1$.

According to \cite[Example 7.9]{Ko14} we have
\begin{equation*}
\begin{aligned}
\Delta(B_{2l})|_{\W^{\otimes 2}}=&B_{2l}\otimes K_{2l}^{-1}+1\otimes F_{2l}+\va_{2l}(q-q^{-1})E_{2l+1}\otimes E_{2l-1}E_{2l}\\
&-\va_{2l}(q^{-1}-q^{-3})E_{2l-1}\otimes E_{2l}E_{2l+1}-\va_{2l} q^{-1}K_{2l-1}K_{2l+1}\otimes E_{2l-1}E_{2l}E_{2l+1}.
\end{aligned}
\end{equation*}
Therefore, we have $B_2\cdot v_{1}\otimes v_{2}=v_{1}\otimes v_{3}$, $B_2\cdot v_{2}\otimes v_{1}=v_{3}\otimes v_{1}$, $B_2\cdot v_{3}\otimes v_{4}=-\va_2 v_{3}\otimes v_{1}$ and $B_2\cdot v_{4}\otimes v_{3}=-\va_2 v_{1}\otimes v_{3}+\va_2(q-q^{-1})v_{3}\otimes v_{1}$. Thus,
\begin{equation}\label{b2cdotv1}
\begin{aligned}
B_2\cdot (v_{1}\otimes v_{2}\cdot \Qy)=&B_2\cdot ( q^{2m-3}(v_{1}\otimes v_{2}-q v_{2}\otimes v_{1})-\va_2^{-1}q^{2m-4}(v_{3}\otimes v_{4}-q v_{4}\otimes v_{3}))\\
=&0,
\end{aligned}
\end{equation}
and $(B_2\cdot v_{1}\otimes v_{2})\cdot \Qy=v_{1}\otimes v_{3}\cdot \Qy=0$. By \eqref{b2cdotv1}, we have $B_2\cdot (v_{a_1}\otimes v_{a_2}\cdot \Qy)=0$. By a direct calculation, we can show that $(B_2\cdot v_{a_1}\otimes v_{a_2})\cdot \Qy=0$. Similarly, we can show that
$$B_{2l}\cdot (v_{a_1}\otimes v_{a_2}\cdot \Qy)=0=(B_{2l}\cdot v_{a_1}\otimes v_{a_2})\cdot \Qy,$$
for any $l=2,\ldots,m-1$. We omit the details.

The proof of the double centralizer property is almost identical to the proof of Proposition \ref{prop:dcp}, which is equivalent to the multiplicity-free decomposition of $\W^{\otimes n}$ as an $\Ui(\spin_{2m})$-$\B_n(-q^{-1},q^{2m})$-bimodule. The proof of the double centralizer property reduces by a deformation argument to the $q=1$ setting. When taking the $q\to 1$ limit and $\va_i=-1$, $\Ui(\spin_{2m})$ becomes the enveloping algebra of the symplectic Lie algebra $\mathfrak{sp}_{2m}$, $\W$ becomes its natural representation, and the multiplicity-free decomposition of $\W^{\otimes n}$ in this case has been established in \cite{Br37}, \cite{Br56a} and \cite{Br56b}. We are done.
\end{proof}

\end{document}